\documentclass{amsart}

\usepackage{amsrefs}
\usepackage{euscript}
\usepackage{amsmath}
\usepackage{amscd}
\usepackage{amssymb}
\usepackage{enumerate}
\usepackage{stmaryrd}
\usepackage{amsfonts}
\usepackage{graphicx}
\usepackage[all]{xy}
\usepackage[margin=1.5in]{geometry}
\usepackage{setspace}
\usepackage{srcltx}
\usepackage{rotating}

\numberwithin{equation}{section}

\newcommand{\bC}{{\mathbb C}}
\newcommand{\bP}{{\mathbb P}}
\newcommand{\bR}{{\mathbb R}}
\newcommand{\bZ}{{\mathbb Z}}

\newcommand{\cF}{\mathcal F}
\newcommand{\cA}{\mathcal A}
\newcommand{\cB}{\mathcal B}

\newcommand{\cH}{\mathcal H}

\newcommand{\sL}{\EuScript L}
\newcommand{\CL}{\mathcal {CL}}
\newcommand{\G}{{\mathcal G}}
\newcommand{\Tq}{T_{q}^{*}Q}
\newcommand{\TQ}{T^{*}Q}
\newcommand{\SQ}{S^{*}Q}
\newcommand{\fold}{\boxbslash}
\newcommand{\fit}{D^{\boxbar}}

\newcommand{\permute}{D^{\boxslash}}
\newcommand{\shuffle}{\bullet }
\newcommand{\btimes}{ \veebar }

\newcommand{\vx}[1][\!]{\vec{x}^{\, #1}}

\renewcommand{\P}{{\mathrm P}}
\newcommand{\F}{{\mathrm F}}
\newcommand{\U}[1][\!]{U^{#1}}
\newcommand{\V}[1][\!]{V^{#1}}

\newcommand{\sJ}{{\EuScript J}}
\newcommand{\sH}{{\EuScript H}}
\newcommand{\Orbit}{{\EuScript O}}
\newcommand{\Chord}{{\EuScript X}}

\newcommand{\bfD}{\mathbf{D}}

\newcommand{\codim}{\operatorname{codim}}
\newcommand{\ev}{\operatorname{ev}}

\newcommand{\Disc}{{\mathcal R}}

\newcommand{\Discbar}{\overline{\Disc}}

\newcommand{\Cyl}{{\mathcal M}}
\newcommand{\Cylbar}{\overline{\Cyl}}
\newcommand{\Ann}{{\mathcal C}}
\newcommand{\Annbar}{\overline{\Ann}}

\newcommand{\OC}{\mathcal {OC}}
\newcommand{\CO}{\mathcal {CO}}
\newcommand{\Wrap}{\EuScript{W}}

\newcommand{\Pil}{\mathcal P}
\newcommand{\Pilbar}{\overline{\Pil}}
\newcommand{\Mod}{\mathcal X}
\newcommand{\Map}{\operatorname{Map}}
\newcommand{\Lef}{C^{\sL}}
\newcommand{\Spin}{Spin}

\renewcommand{\dbar}{\overline{\partial}}
\def\co{\colon\thinspace}

\newtheorem{thm}{Theorem}[section]
\newtheorem{cor}[thm]{Corollary}
\newtheorem{lem}[thm]{Lemma}
\newtheorem{prop}[thm]{Proposition}
\newtheorem{defin}[thm]{Definition}
\newtheorem{def-lem}[thm]{Definition-Lemma}

\theoremstyle{remark}

\newtheorem{rem}[thm]{Remark}

\newtheorem{example}[thm]{Example}

\newcommand{\superscript}[1]{\ensuremath{^{\textrm{#1}}} }

\renewcommand{\th}[0]{\superscript{th}}
\newcommand{\st}[0]{\superscript{st}}

\newcommand{\noproof}{\qed}

\newcommand{\comment}[1]{}

\title[A cotangent fibre generates the Fukaya category]{A cotangent fibre generates the Fukaya category}

\author[M.~Abouzaid]{Mohammed Abouzaid} \date{\today}
\thanks{This research was conducted during the period the author served as a Clay Research Fellow. }

\begin{document}
\begin{abstract}
We prove that the algebra of chains on the based loop space recovers the derived (wrapped) Fukaya category of the cotangent bundle of a closed smooth oriented manifold.  The main new idea is the proof that a cotangent fibre generates the Fukaya category using a version of the map from symplectic cohomology to  the homology of the free loop space introduced by Cieliebak and Latschev. 
\end{abstract}
\maketitle
\setcounter{tocdepth}{1}
\tableofcontents

\section{Introduction}
In this paper, we prove that the wrapped Fukaya category of a cotangent bundle is expressible in terms of purely homotopy-theoretic data:
\begin{thm} \label{thm:main_thm}
If $Q$ is an oriented closed smooth manifold, then any cotangent fibre generates the wrapped Fukaya category of $T^* Q$ with background class $b \in  H^{*}(\TQ, \bZ_{2})$ given by the pullback of the second Stiefel-Whitney class of $Q$.  Moreover, the triangulated closure of this Fukaya category is quasi-isomorphic to the category of twisted complexes over $C_{-*} (\Omega_{q} Q) $.   \end{thm}
\begin{rem}
In \cite{nadler}, Nadler shows that a different version of the Fukaya category of a cotangent bundle which he constructed with Zaslow in \cite{NZ} whenever $Q$ is real analytic, is equivalent to the category of constructible sheaves on $Q$.  
\end{rem}
\begin{rem} \label{rem:sign-error}
The result uses the existence of coherent orientations of moduli spaces of holomorphic discs with boundary on any collection of Lagrangians whose second Stiefel-Whitney class is the restriction of the same background class in the cohomology of the total space, as well as the identification of an appropriate twist of the symplectic cohomology of $\TQ$ with the homology of the free loop space of $Q$.  The fact that the untwisted version of symplectic cohomology is not in general isomorphic to the homology of the free loop space was verified by Seidel for the case of $T^* \bC \bP^2$ in \cite{seidel:cp2}, in an attempt to elucidate the source of a sign discrepancy between the construction of a Viterbo restriction map on symplectic cohomology \cite{viterbo}, and a generalisation established by Kragh in \cite{kragh} using generating functions.  There is also a version of Theorem \ref{thm:main_thm}, stating that the \emph{untwisted} wrapped Fukaya category of $\TQ$ is also generated by a fibre, and is equivalent to the category of modules over the chains of the based loop space of $Q$ with twised coefficients (see Remark 1.2 in \cite{string-top}). 
\end{rem}
Note that as a consequence of the above equivalence, we conclude that the Grothendieck $K$-group of the wrapped Fukaya category of $\TQ$ is free of rank $1$, and is generated by the class of a cotangent fibre.  Since the zero section intersects the fibre in exactly one point, we find that the homomorphism 
\begin{equation*}
  K_{0}(\Wrap(\TQ)) \to \bZ
\end{equation*}
is realised by taking the Euler characteristic of the space of morphisms to (or from) the zero section. 

The fact that the wrapped Fukaya category is \emph{generated} rather than split-generated does not follow from the machinery of \cite{generation}.  Rather, it is a consequence of the existence of 
an $A_{\infty}$ homomorphism $\cF$  from the wrapped Floer cochain complex of a cotangent fibre to the Pontryagin differential graded algebra $C_{-*} (\Omega_{q} Q)  $ of chains on the based loop space, which was constructed in \cite{string-top}.  On homology this homomorphism induces a map
\begin{equation}
\label{eq:map_wrapped_Floer_based_loop}
H^{*}(\cF) \co  HW^{*}_{b}(\Tq) \to H_{-*} (\Omega_{q} Q).
\end{equation}
This map has a closed string analogue from symplectic cohomology to the homology of the space of free loops
\begin{equation}
  \label{eq:map_SH_free_loops}
H^{*}(\CL) \co SH^{*}_{b}(\TQ) \to H_{n-*}(\sL Q),
\end{equation}
which is compatible with the grading of both sides by the set of components of the free loop space.  Such a map was first proposed by Cieliebak and Latschev in \cite{CL} who used it to compare algebraic structures in Symplectic Field theory with those coming from String topology. In Section \ref{sec:count-half-cylind}, we give a realisation of this map in the setting of Floer theory.  Note that this is one place where the orientability assumption is used: in general, the twisted version of symplectic cohomology that we consider is isomorphic to the homology of the free loop space with coefficients in the orientation bundle of the base, pulled back by evaluation at the basepoint.
\begin{prop} \label{prop:maps_iso}
 $ H^{*}(\cF) $ and $ H^{*}(\CL) $ are both isomorphisms. \noproof
\end{prop}
\begin{rem}
The statement of this Proposition was communicated to the author by Schwarz in Summer 2009 as an announcement of results obtained and to be written jointly with Abbondandolo.  Subsequently, a sketch of the proof for $ H^{*}(\cF)  $ was included in Section 5 of \cite{string-top}.  Given the nature of the construction, one can use the same method to show that $ H^{*}(\CL)  $ is an isomorphism, and we briefly discuss the relevant signs in Appendix \ref{sec:signs-orientations}.
\end{rem}

In addition to Proposition \ref{prop:maps_iso}, the proof of Theorem \ref{thm:main_thm} relies essentially on the results of \cite{generation}, which defines a map from the Hochschild homology of the $A_{\infty}$ algebra $CW^{*}_{b}(\Tq)  $  to symplectic cohomology:
\begin{equation}
  \label{eq:open_closed}
H^{*}(  \OC)  \co HH_{*}( CW^{*}_{b}(\Tq))  \to SH^{*+n}_{b}( \TQ)
\end{equation}
which we review in Section \ref{def:floer_datum_discs_interior_puncture}.

In Section \ref{sec:an-ad-hoc}, we construct a map
\begin{equation}
  \label{eq:homology_hochschild_loops}
H^{*}( \G)  \co HH_{*}(C_{-*} (\Omega_{q} Q) )  \to  H_{-*}(\sL Q)
\end{equation}
which we expect to be an isomorphism since it should be a version of Goodwillie's  isomorphism from \cite{good}.  While we do not prove this, we shall prove in Appendix \ref{sec:hitting} that the fundamental class of $Q$, included as constant loops in the homology of the free loop space,  lies in the image of $ H^{*}( \G) $.   In Lemma \ref{lem:constants_factor}, we show that $H^{*}(  \CL)  $ maps the identity of symplectic cohomology to this fundamental class.
\begin{prop}
  \label{prop:commutative_diagram}
The following diagram commutes up to sign:
\begin{equation}
  \label{eq:commutative_diagram_symplectic_string}
 \xymatrixcolsep{5pc} \xymatrix{ HH_{*}( CW^{*}_{b}(\Tq))  \ar[r]^{H^{*}(\OC)}  \ar[d]^{HH_{*}(\cF)} & SH^{*+n}_{b}( \TQ)   \ar[d]^{H^*(\CL)} \\
 HH_{*}(C_{-*} (\Omega_{q} Q) )   \ar[r]^{H^{*}(\G)}  &  H_{-*}(\sL Q). }
\end{equation}
\end{prop}
Theorem \ref{thm:main_thm} is now a rather direct consequence of the results proved in \cite{generation}.
\begin{proof}[Proof of Theorem  \ref{thm:main_thm}]
If  $ H^{*}(\cF) $ is an isomorphism, then so is the map induced by $\cF$ on Hochschild homology.  Knowing the two vertical arrows are isomorphisms and that the identity of $SH^{*}_{b}(\TQ)$ maps to the fundamental class of $Q$ under $ H^{*}(  \CL)  $, the commutativity of Diagram \eqref{eq:commutative_diagram_symplectic_string}, together with Lemma  \ref{lem:fundamental_class_in_image}, implies that the identity in symplectic cohomology lies in the image of $ H^{*}( \OC) $.  By Theorem 1.1 in \cite{generation}, we conclude that $\Tq$ split-generates the wrapped Fukaya category of the cotangent bundle.

To pass from split-generation to generation, we note that Corollary 1.2 in \cite{string-top} extends the $A_{\infty}$-homomorphism $\cF$ to a functor from the wrapped Fukaya category of $\TQ$ to the category of twisted complexes over $ C_{-*} (\Omega_{q} Q)$.   Since $\Tq$ split-generates the wrapped Fukaya category, this is a cohomologically fully faithful embedding, and hence every object of the wrapped Fukaya category of $\TQ$ is in fact isomorphic to an iterated cone of cotangent fibres.
\end{proof}

\begin{rem}
At first sight, our claim about the existence of a natural map from Hochschild homology to symplectic cohomology would seem to indicate that we failed to account for a dualisation, or at least to properly name one of the two groups.  The reason for the confusion is the fact that, while there is a natural map in the direction we indicated, there is also another from symplectic cohomology to Hochschild cohomology.  In the case of wrapped Fukaya categories of sufficiently nice manifolds (i.e. ones with enough Lagrangians), both of these maps are expected to be isomorphisms, and hence Hochschild homology and cohomology are isomorphic.  The isomorphism between them is expected to be part of the \emph{Calabi-Yau} structure on the wrapped Fukaya category, and is sufficiently non-trivial that its existence (in this setting) has not yet been proved.
\end{rem}

\subsection*{Acknowledgments}
I would like to thank Ronald Brown for pointing out Barr's work \cite{barr}, and Kate Ponto for helpful comments on a draft version of Appendix \ref{sec:cuboidal-chains}.  Much of this paper was written while the author visited MSRI during the 2009-10 program.   I would also like to thank Paul Seidel, Thomas Kragh, as well Alberto Abbondandolo, and Matthias Schwarz for discussions about the sign discussed in Remark \ref{rem:sign-error}; of course I am responsible for any and all remaining sign mistakes and misinterpretations of other people's sign conventions.   The final version of the paper benefited from useful comments from an anonymous referee.
\section{The open sector} \label{sec:open-sector}
Given a compact connected smooth manifold $Q$, the cubical chain complex of the space of Moore loops based at a point $q$  forms a differential graded algebra $C_{-*}(\Omega_{q} Q)$  where multiplication is induced by concatenation of paths.  To turn this into an $A_{\infty}$ structure, we use the conventions:
\begin{align*} 
\mu^{\P}_1 \sigma & \equiv  \partial \sigma \\
\mu_{2}^{\P}(\sigma_{2}, \sigma_{1}) & \equiv (-1)^{|\sigma_1|} \sigma_{1} \cdot \sigma_{2},
\end{align*}
where $|\sigma_1| = - \dim(\sigma_1) $.  In \cite{string-top} we constructed an $A_{\infty}$ homomorphism from the wrapped Floer cochains of a cotangent fibre to this algebra.  This section contains no new results, but it is instead meant to briefly review the notation \cite{string-top}, slightly simplified because we shall consider a Fukaya category consisting of only one object.  We shall also use $T^* S^1$  to illustrate the general construction.
\subsection{Geometric preliminaries} \label{sec:geom-prel}
Fix a Riemannian metric on $Q$, and let $\cH(\TQ)$  denote the space of smooth functions which agree with $|p|^2$ whenever $|p| \geq 1$ (here, we assume $Q$ is locally given coordinates $q_i$, with $p_i$ the corresponding coordinates of the cotangent fibre, and $|p|^2$ is shorthand for $\sum_{i=1}^{\dim(Q)} |p_i|^2$).  The cotangent bundle is equipped with the canonical Liouville $1$-form $\lambda$, whose differential is a symplectic form denoted $\omega$, and with a quadratic complex volume form obtained by complexifying an (ordinary) volume form on $Q$. We write $H(q,p) = |p|^2$ with Hamiltonian flow $X$, and assume that the following generic condition holds
\begin{equation}
  \label{eq:genericity_chords_orbits}
   \parbox{36em}{All Reeb orbits on the contact hypersurface  $\SQ$ where $|p|=1$ and all  flow lines of $X$ of time $1$ with boundary on $\Tq $ are non-degenerate. }
\end{equation}
We write $\Chord$ for the set of such flow lines which are called time-$1$ chords.  Since the complexification of a (real) volume form on $Q$ defines a complex-valued volume form on $\TQ$, we may assign to each chord $x$ a Maslov index we denote $|x|$, and a path $\Lambda_{x}$ of Lagrangians in $x^{*} \left( T \TQ\right) $ which agrees at either end with the tangent space to the fibre at $q$, and is uniquely determined up to homotopy by the property that the induced map
\begin{equation}
  \label{eq:path_Lagrangian_chord}
  S^1 \to \bC^{*}
\end{equation}
obtained by evaluating the square  of the holomorphic volume form on a frame of $\Lambda_{x}(t)$ is contractible.   As in Section (11l) of \cite{seidel-book}, one uses $\Lambda_{x}$ to define an elliptic operator $D_{x}$ on a disc with one puncture, whose determinant line we denote $o_x$.  The wrapped Floer complex has underlying graded abelian group
\begin{equation}
  \label{eq:wrapped_complex}
  CW^{i}_{b}(\Tq) = \bigoplus_{\substack{|x|= i \\ x \in \Chord}} | o_x|
\end{equation}
where $|o_x|$ is the rank $1$ free abelian group generated by the possible orientations of $o_x$ with the relation that the sum of opposite orientations vanishes.  The same construction can be performed at the intersection point $q$ of $\Tq$ and $Q$: we obtain a path $\Lambda_{q}$ of linear Lagrangians in $T \TQ| q$ starting at the tangent space of the zero section and ending at the cotangent fibre and write $o_{q}$ for the determinant line of the corresponding operator.

\begin{rem}
The reader who does not want to be burdened with keeping track of signs should instead think that $CW^{i}_{b}(\Tq)  $ is the abelian group freely generated by chords of Maslov index $i$.    
\end{rem}

In order to orient moduli space of holomorphic curves, we consider a vector bundle $E_{b}$ on $\TQ$ which is isomorphic to the pullback of the tangent bundle of $Q$.  On $\Tq$ and $Q$ we choose a \emph{relative $\Spin$ structure}.  Letting $L$ stand of either of these Lagrangians, such a structure is defined to be
\begin{equation}
  \label{eq:rel_pin_structure}
  \parbox{36em}{a $\Spin$ structure on the direct sum of $TL$ with the  restriction of $E_{b}$.}
\end{equation}
The obstruction to the existence of such a structure is the second Stiefel-Whitney class of the direct sum, which vanishes in one case because $\Tq$ is contractible, and in the other because it is equal to twice the Stiefel-Whitney class of $Q$.  For each chord $x$ and at the intersection point $q$ we also choose a \emph{relative $\Spin$ structure} which consists of a
\begin{equation}
  \label{eq:Pin_structure_chord}
  \parbox{36em}{a $\Spin$ structure on $\Lambda_x \oplus x^{*}(E_b)$ which restricts at the endpoints to the $\Spin$ structure on $\Tq \oplus E_{b}$.}
\end{equation}

Let $\sJ(\TQ) $ denote the space of almost complex structures on $\TQ$ which are compatible with $\omega$, and whose restriction to the complement of a compact set is of contact type in the sense that
\begin{equation*}
 \lambda \circ J = dr
\end{equation*}
whenever $J \in \sJ(\TQ) $.  Consider a family $I_t$ of such structures parametrised by the interval $[0,1]$ as well as a map 
\begin{equation}
  \tau \co [0,1] \to [0,1]
\end{equation}
which agrees with the identity on the boundary, and is locally constant in a neighbourhood thereof.

\begin{example}
 It is useful to keep in mind that the elements of $\Chord$ are in bijective correspondence with intersection points between $\Tq$ and its image under the time-$1$ Hamiltonian flow of $H$.  In the case $Q = S^1$, the top picture in Figure \ref{fig:cylinder-wrap} shows the cotangent fibre and its image under the flow.  Note that there is exactly one intersection point, and hence one chord, in each relative homotopy class of based paths on $S^1$.  After choosing an orientation for $S^1$, we may associate an integer to each such chord, corresponding to the number of times it winds around the circle.  All these chords have Maslov index $0$ with the standard choice of complex volume form on $ T^* S^1$, which, upon identification with $\bC^{*}$, takes the form $\frac{dz}{z}$.
\end{example}

\begin{figure}
  \centering
  \includegraphics{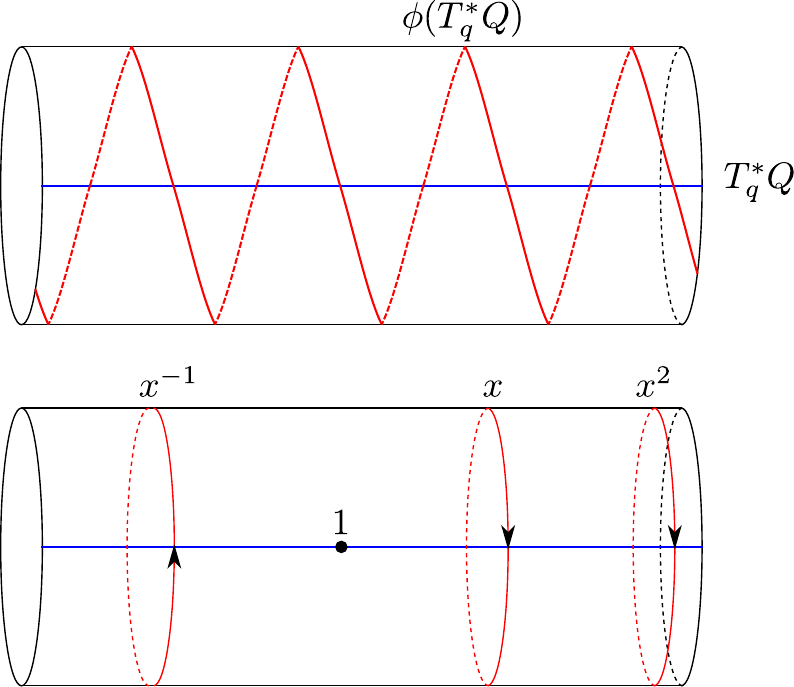}
  \caption{ }
  \label{fig:cylinder-wrap}
\end{figure}

\subsection{The wrapped Floer complex}
Given a pair $x_0, x_1$  of distinct elements of $\Chord$, we define $\Disc(x_0;x_1)  $ to be the quotient, by the $\bR$ action which translates the first variable, of the space of maps
\begin{equation*}
  u \co (-\infty,+\infty) \times [0,1] \equiv Z \to \TQ
\end{equation*}
taking the boundary to $\Tq$, that converge exponentially at $-\infty$ to $x_0$ and at $+\infty$ to $x_1$ and that satisfy Floer's equation
\begin{equation}
  \label{eq:dbar_strip}
  \left(du - X \otimes d\tau  \right)^{0,1} = 0.
\end{equation}

Assuming that $I_t$ has been chosen generically, the moduli spaces $\Disc(x_0;x_1) $ are smooth manifolds of dimension
\begin{equation*}
  |x_0| - |x_1| -1.
\end{equation*}
Whenever $|x_0| =  |x_1| +1 $, we conclude that all elements $u$ of $\Disc(x_0;x_1) $  are rigid. Using the choice of $\Spin$ structure fixed in \eqref{eq:Pin_structure_chord}, the standard argument proving invariance of the index under Fredholm deformations implies that every such rigid map $u$ defines a canonical isomorphism up to homotopy
\begin{equation} \label{eq:iso_orientation_differential}
   o_{x_1} \to o_{x_0},
\end{equation}
as reviewed in Appendix  \ref{sec:signs-orientations}.

Writing $\mu_u$ for the induced map on orientation lines, we define 
\begin{align}
  \label{eq:wrapped_differential}
  \mu_{1} \co  CW^{i}_{b}(\Tq)  & \to CW^{i+1}_{b}(\Tq)  \\
[x_1] & \mapsto (-1)^i \sum_{u} \mu_u([x_1] ).
\end{align}
\begin{example}
  On $T^*S^1$,  $\mu_1$ vanishes identically since each chord lies in a different relative homotopy class. The vanishing of $\mu_1$ may also be proved using the fact that $\mu_1$ has degree $1$, while all generators have degree $0$.
\end{example}

\subsection{The $A_{\infty}$ structure}
The $A_{\infty}$ structure on $   CW^{*}_{b}(\Tq)  $ is defined by counting maps whose sources are elements of the compactified moduli space  $\Discbar_{d}$ of discs with one outgoing boundary marked point and $d$ incoming ones.  A point in the smooth part $\Disc_{d}$ is obtained by taking the complement $S$ of $d+1$ points removed from the boundary of a closed disc, with one of them distinguished as outgoing; starting with the outgoing point, we can order them counterclockwise $(\xi^0, \xi^{1}, \ldots, \xi^{d})$. We choose a \emph{negative end}  near the outgoing marked point, i.e. a holomorphic map from a negative half-strip
\begin{equation*}
  (-\infty, 0) \times [0,1]  \to S 
\end{equation*}
which converges at $-\infty$ to $\xi^0$ and parametrises a neighbourhood thereof.  At the incoming marked points we choose \emph{positive ends}, which are parametrised instead by the positive half-strip.

These choices of ends can be made smoothly with respect to the modulus of the curve, allowing us to construct charts near the corner strata of $  \Discbar_{d}$.  Recall that a stratum $\sigma \subset  \partial \Discbar_{d}$ of codimension $k$ is represented by curves with $k$ distinct components arranged along a tree; each component can be thought of as lying in a moduli space $ \Disc_{d'}  $ for $d' < d$.  In particular, there are two strip-like ends, one positive, the other negative, associated to each node. For each parameter  $R \in (0,+\infty]$, we obtain a new Riemann surface by removing the images of $(-\infty, -R)$ and $(R,+\infty)$ for the two ends, and gluing the complements.  If we perform this construction at every node, we obtain a map
\begin{equation} \label{eq:chart_corner_moduli_space}
  \sigma \times (0, +\infty]^{k} \to  \Discbar_{d}
\end{equation}
which is a local homeomorphism near infinity.  Note that we are parametrising this chart by the gluing parameter, so that $ \sigma \times \{+\infty\}^{k}$ corresponds to the corner stratum.

Following \cite{seidel-book}, we shall not solve the actual $\dbar$ equation on elements of these moduli spaces, but rather perturbations thereof, which are allowed to depend on the modulus of the curve.
\begin{defin} \label{def:floer_datum_disc_1_output}
A  \emph{Floer datum} $D_{S}$ on a stable disc $S \in \Discbar_{d}$  consists of the following choices on each component:
\begin{enumerate}
\item Time shifting map:  A map $ \rho_{S} \co \partial S \to [1,+ \infty)$ which is constant near each end.  We write $w_{k,S}$ for the value on the $k$\th end. 
\item Basic $1$-form and Hamiltonian perturbations:  A closed $1$-form $\alpha_{S}$ whose restriction to the boundary vanishes, and a map $H_{S} \co S \to \sH(\TQ)$  on each surface defining a Hamiltonian flow $X_{S}$.  The pullback of $ X_{S} \otimes \alpha_{S} $ under the $k$\th end should agree with  
  \begin{equation*}
    X_{\frac{H}{w_{k,S}} \circ \psi^{w_{k,S}}} \otimes d\tau.
  \end{equation*}
\item Almost complex structures:  A map $I_{S} \co S \to \sJ(\TQ)$ whose pullback under the $k$\th end agrees with  $ (\psi^{w_{k,S}})^{*} I _t$.
\end{enumerate}
\end{defin}
This data allows us to write down the Cauchy-Riemann equation 
\begin{equation}
  \label{eq:dbar_disc_1_output}
\left(du - X_{S} \otimes \alpha_{S}\right)^{0,1} = 0
\end{equation}
on the space of maps from $u$ to $\TQ$.  In order for counts of solutions to this equation to define operations that satisfy the $A_{\infty}$ condition, we must choose these perturbations in a sufficiently compatible way for all possible Riemann surfaces $S$.  

We say that two such choices of data $\left( \rho_{S}^{1}, \alpha_{S}^{1}, H_{S}^{1}, I_{S}^{1} \right) $ and  $\left( \rho_{S}^{2}, \alpha_{S}^{2}, H_{S}^{2}, I_{S}^{2} \right) $ are \emph{conformally equivalent} if there exists a constant $C$ so that $\rho^{2}_{S}$ and $\alpha_{S}^{2}$ respectively agree with   $C \rho^{1}_{S}$ and $C \alpha_{S}^{1}$, and
\begin{align*}
 I_{S}^{2} & = {\psi^{C}}^{*} I_{S}^{1} \\
H_{S}^{2} & = \frac{H_{S}^{1} \circ \psi^{C}}{C^{2} }.
 \end{align*}

\begin{defin}
A \emph{universal and conformally consistent}  choice of Floer data for the $A_{\infty}$ structure,  is a choice $\bfD_{\mu}$  of such Floer data  for every integer $d \geq 2$, and every (representative of an) element of  $\Discbar_{d}$.  We require that these data vary smoothly over the compactified moduli space and that their restrictions to a boundary stratum be conformally equivalent to those coming from lower dimensional moduli spaces.  Finally,  near a boundary stratum the Floer data should agree to infinite order in the coordinates \eqref{eq:chart_corner_moduli_space} with the data obtained by gluing.
\end{defin}
Given a fixed generic universal and conformally consistent choice of  Floer data $\bfD_{\mu}$, we define a map
\begin{equation*}
  \mu_d \co  CW^{*}_{b}(\Tq)^{\otimes d}  \to   CW^{*}_{b}(\Tq)
\end{equation*}
using the moduli spaces $\Disc_{d}(x_0, \vx)$  of solutions $u$ to Equation \eqref{eq:dbar_disc_1_output} on a disc $S \in \Discbar_{d}$ with respect to $\bfD_{\mu}$, with boundary condition $\Tq$, and which converge to $x_0$ at the negative end, and to  $\vx=\{ x_1, \ldots, x_d\}$ at the positive ends.   As we briefly recall  in Appendix \ref{sec:signs-orientations} from \cite{string-top} the choices of relative $\Spin$ structures determine an isomorphism
\begin{equation} \label{eq:isomorphism_disc}
 \lambda( \Disc_{d}(x, \vx) )  \otimes o_{x_d} \otimes \cdots \otimes o_{x_1} \cong  \lambda( \Disc_{d}) \otimes o_{x_0},
\end{equation}
where $\lambda$ stands for the top exterior power of the tangent bundle.  Whenever $|x| =  2 - d + \sum_{1 \leq k \leq d} |x_k| $, the moduli space $ \Disc_{d}(x; \vx)   $  is rigid.  In particular, we obtain an isomorphism
\begin{equation*}
    o_{x_d} \otimes \cdots \otimes o_{x_1} \to o_{x}
\end{equation*}
from an orientation of $  \Disc_{d} $.  Our orientation on $\Disc_{d}$, following Section (12g) of \cite{seidel-book}, uses its identification with the configuration space of $d-2$ points $\xi^{3}, \ldots, \xi^{d}$ on an interval.   We let $\mu_u$ denote the map induced on orientation lines, and define 
\begin{equation}
 \mu_{d}([x_d], \ldots, [x_1])  = \sum_{u \in \Disc_{d}(x; \vx) } (-1)^{\dagger} \mu_{u}( [x_d], \ldots, [x_1])
\end{equation}
where the sign is given by
\begin{equation} \label{eq:dagger_sign}
 \dagger = \sum_{k=1}^{d} k |x_k|.
\end{equation}

\begin{example}
On the wrapped Floer complex of a cotangent fibre in $T^*S^1$, the higher products $\mu_d$ vanish if $d \geq 3$ because they have degree $2-d$, while all the generators have degree $0$.   It is unfortunately inconvenient to see the product $\mu_2$ if we think of the Floer complex as generated by chords.  However, using the equivalent model where the Floer complex is generated by intersection points between a cotangent fibre and its image under the time-$1$ Hamiltonian flow $\phi$ of $H$, one may express $\mu_2$ as a product 
\begin{equation} \label{eq:product_apply_phi}
CF^{*}(\phi( \Tq) ,  \Tq) \otimes   CF^*(\phi^{2} (\Tq), \phi( \Tq)) \to   CF^*(\phi^2 (\Tq), \Tq) ,
\end{equation}
which, in favourable circumstances, can be obtained by counting rigid \emph{holomorphic} curves.  In the case of $T^*S^1$,  Figure \ref{fig:cylinder-wrap-product} shows the image of the cotangent fibre under $\phi$ and $\phi^2$, as well as the disc which proves that $ \mu_{2}(1,x) = x $ (note that in the picture, the labels appear to be ordered clockwise; this is an artefact of our choice of symplectic form on the cotangent bundle).
\end{example}
\begin{figure}
  \centering
  \includegraphics{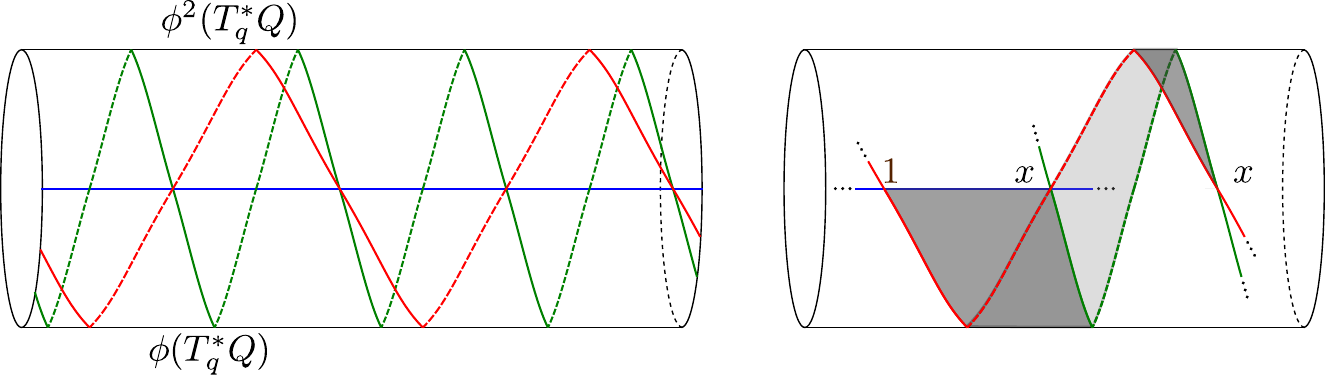}
  \caption{ }
  \label{fig:cylinder-wrap-product}
\end{figure}

\subsection{The moduli space of half-discs} \label{sec:moduli-space-half}
In this section, we shall define the moduli spaces which give rise to the $A_{\infty}$ homomorphism $\cF$ discussed in the introduction.  This is essentially a review of the results from Section $4$ of \cite{string-top}, with a few simplifying features coming from the fact the cotangent fibre intersects the zero section at only one point.

Let $\Pil_{d}$  denote the moduli space of holomorphic discs with $d+2$ boundary punctures, of which $d$ successive ones $\{ \xi^1, \ldots, \xi^d \} $ are distinguished as incoming; the segment connecting the remaining marked points $\{ \xi^{-1}, \xi^0 \}  $ is called the outgoing segment.  We shall call an element of $\Pil_{d}$ a \emph{half-disc}.  We identify $\Pil_{0}$ with a point (equipped with a group of automorphisms isomorphic to $\bR$) corresponding to the moduli space of strips. In addition, we fix an orientation on the moduli space $\Pil_{d}$ using the conventions for Stasheff polyhedra and the isomorphism
\begin{equation}\label{eq:half_discs_are_discs} \Pil_{d} \cong \Disc_{d+1}  \end{equation}
taking the incoming marked points on the source to the first $d$ incoming marked points on the target. 

The Deligne-Mumford compactification $ \Pilbar_{d}$ is simply a copy of  $\Discbar_{d+1}$, but the operadic structure maps associated to the boundary strata are different. If breaking takes place away from the outgoing segment, the domain is determined by sequences $\{ 1 , \ldots, d_1 \}$ and $\{ 1, \ldots, d_2\}$, such that $d_1 +d_2 = d+1$, and a fixed element $k$ in the first sequence.  By gluing the outgoing end of an element of $\Discbar_{d_2}$ to the $k+1$\st incoming end of a half disc, we obtain a map 
\begin{equation}\label{eq:half_popsicle_break_top}\Pilbar_{d_1} \times \Discbar_{d_2} \to \Pilbar_{d}  \qquad d_2=d-d_1+1. \end{equation}
If breaking occurs on the outgoing segment, it is determined by a partition $\{1, \ldots, d \} = \{1, \ldots, d_1 \} \cup \{ d_1 +1, \ldots, d \}$.  By gluing the $0$\th end of a half-disc with $d_2=d-d_1 $ inputs with the end labelled $\xi_{-1}$ of a half-disc with $d_1$ inputs, we obtain a map 
\begin{equation}\label{eq:half_popsicle_break_middle} \Pilbar_{d_1} \times  \Pilbar_{d_2} \to \Pilbar_{d}  \qquad d_2=d-d_1.  \end{equation}
The constant $1$ appears in Equation \eqref{eq:half_popsicle_break_top} but not in \eqref{eq:half_popsicle_break_middle} because the label records the number of incoming points on the boundary; our conventions are such that neither of the two marked points on the outgoing arc of an element of $ \Pilbar_{d}$ is incoming.

The $A_{\infty}$ homomorphism $\cF$ uses moduli spaces of solutions to a family of equations parametrised by $\Pilbar_{d}$.  Note that the isomorphism \eqref{eq:half_discs_are_discs}, and the choice of strip-like ends on elements of $\Discbar_{d+1}$, equips an element of $\Pilbar_{d}$  with strip-like ends near each puncture: the $d$ incoming ends as well as $\xi^{-1}$ acquire positive ends, while we have a negative end near $\xi^{0}$.

Let us also, in addition to the function $H$ chosen earlier, fix 
\begin{equation}
  \parbox{36em}{a function $G \in \cH(\TQ)$ which vanishes identically near the zero section. }
\end{equation}
\begin{defin} \label{def:floer_datum_half_popsicle}
A  \emph{Floer datum} $D_{T}$ on a stable disc $T \in \Pilbar_{d}$  consists of the following choices on each component:
\begin{enumerate}
\item Time shifting map:  A map $ \rho_{T} \co \partial T \to [1,+ \infty)$ which is constant near each end and is equal to $1$ on the outgoing segment.  We write $w_{k,T}$ for the value on the $k$\th end. 
\item Basic $1$-form:  A closed $1$-form $\alpha_{T}$  whose restriction to the complement of the outgoing segment in $\partial T$ and to a neighbourhood of  $\xi^0$ and $\xi^{-1}$  vanishes, and whose pullback under  the $k$\th end agrees with  $ w_{k,T} d \tau $.
\item Hamiltonian perturbation: A map $H_{T} \co T \to \sH(\TQ)$  on each surface such that the restriction of $H_{T}$ to a neighbourhood of the outgoing boundary segment agrees with $G$.  We write  $X_{T}$  for the Hamiltonian flow of $H_T$ and assume in addition that the pullback of $ X_{T} \otimes \alpha_{T} $ under the $k$\th end agrees with  $  X_{\frac{H}{w_{k,T}} \circ \psi^{w_{k,T}}} \otimes dt$ if $1 \leq k \leq d$.
\item Almost complex structures:  A map $I_{T} \co T \to \sJ(\TQ)$ whose pullback under the $k$\th end agrees with  $ (\psi^{w_{k,T}})^{*} I _t$.
\end{enumerate}

A   \emph{universal and conformally consistent}  choice of Floer data for the homomorphism $\cF$  is a choice $\bfD_{\cF}$  of such Floer data  for every integer $d \geq 1$ and every (representative) of an element of  $\Pilbar_{d}$ which varies smoothly over this compactified moduli space.   The restriction of $ \bfD_{\cF} $   to a boundary stratum should be conformally equivalent to the product of Floer data coming from either $\bfD_{\mu}$ or a lower dimensional moduli space  $\Pilbar_{d}$, and,  near such a boundary stratum,  should  agree to infinite order with the Floer data obtained by gluing.
\end{defin}
The stratification of the boundary of   $\Pilbar_{d}$ gives a procedure for constructing Floer data inductively.  The choice on the unique point $T_1 \in \Pilbar_{1}$ is subject only to the constraints of the first half Definition \ref{def:floer_datum_half_popsicle}.    Having fixed such data, gluing two curves in $\Pilbar_{1}$ defines Floer data on a neighbourhood of one of the boundary strata of $\Pilbar_{2}$, while gluing the  data for $T_1$ to the result of rescaling  the restriction of $\bfD_{\mu}$ to $\Discbar_{2}$ by $w_{1,T_1}^{-1}$ defines data near the other boundary component.  We choose perturbations of these two glued data which vanish to infinite order at the boundary, then extend these choices to the rest of the moduli space $\Pilbar_{2}$.  These steps are then repeated for every integer $d \geq 2$.

Let us now fix a collection $\vx = \{ x_1, \ldots, x_d \}$ of chords with boundary on $\Tq$, and define $\Pil(q,\vx,q) $ to be the moduli space of finite energy maps 
\begin{equation*}
  u \co T \to \TQ
\end{equation*}
for an arbitrary element $T$ of $\Pil_{d}$, with the outgoing segment mapping to $Q$, all other components mapping to $\Tq$, asymptotic conditions $\vx$ along the incoming ends, and satisfying the differential equation
\begin{equation}
  \label{eq:dbar_half_discs}
  \left(du - X_{T} \otimes \alpha_{T}\right)^{0,1} = 0
\end{equation}
with respect to the $T$-dependent almost complex structure $I_{T}$.
\begin{lem} \label{lem:moduli_half_discs_manifold_corners}
 For generic data  $\bfD_{\cF}$,  the moduli space $\Pil(q,\vx,q) $ is a smooth manifold of  dimension 
  \begin{equation}
   d-1 -  \sum |x_i|
  \end{equation}
whose Gromov bordification is a compact manifold with boundary.  The boundary is covered by the closures of the codimension $1$ strata
\begin{equation} \label{eq:moduli_space_maps_breaking_side}
  \Pil(q,\vx[1],q)  \times   \Pil(q,\vx[2],q) 
\end{equation}
for a partition $ \vx[1] = \{ x_1, \ldots, x_{d_1} \}$ and  $\vx[2] = \{x_{d_1+1}, \ldots, x_d  \} $, and
\begin{equation} \label{eq:moduli_space_maps_breaking_top}
  \Pil(q,\vx[1],q)  \times   \Disc(x;\vx[2]) 
\end{equation}
where $x$ is one of the elements of $\vx[1]$, and $\vx$ is obtained by replacing this element by the sequence $\vx[2]$.
\end{lem}
\begin{proof}
Transversality is a standard consequence of the Sard-Smale argument.  To prove compactness, choose a positive real number $r$ sufficiently large that no element of $\vx$ intersects  $\SQ \times [r,+\infty)$, and let $T'$ denote the inverse image of this region under an element of $\Pil(q,\vx,q) $.  Since the outgoing boundary segment is mapped to the zero section which is disjoint from $\SQ \times [r,+\infty)$, the restriction of $\alpha_T$ to $T'$ vanishes on all the boundary components with Lagrangian labels.  In particular, the hypothesis of  Lemma A.1 in  \cite{generation} holds, so that $u|T'$ is constant.  The result now follows from the standard methods of Gromov compactness.
\end{proof}

\begin{example}
 On $T^* S^1$ the moduli spaces $\Pil(q,\vx,q) $ can only be rigid whenever $\vx$ is a sequence with exactly one element.  One may choose the Floer data so that $\Pil(q,x^{i},q)$ consists of exactly one element for each chord.  If $|i| > 1$, then the corresponding curve multiply covers some part of $ T^* S^1 $, but for $x$ and $x^{-1}$, the image of the curve is an annulus, which is cut by the cotangent fibre into a rectangle (see Figure \ref{fig:cylinders_half_discs}).  
\end{example}
\begin{figure}
  \centering
  \includegraphics{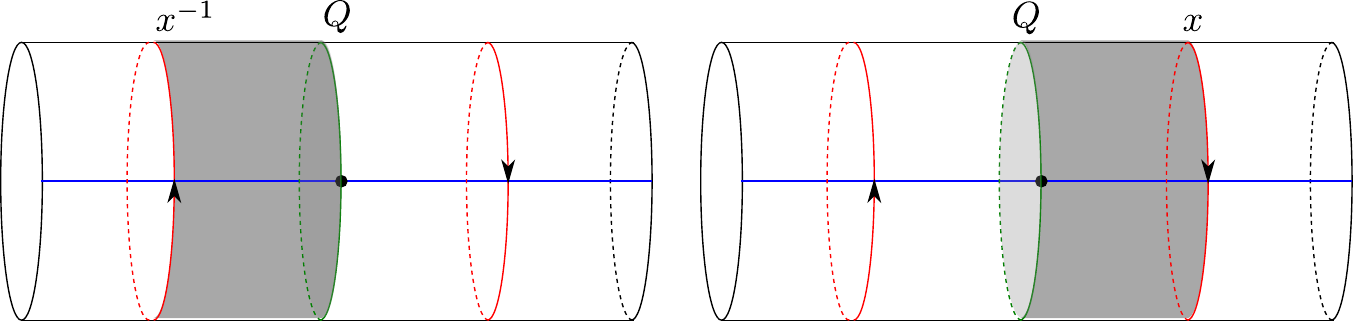}
  \caption{ }
  \label{fig:cylinders_half_discs}
\end{figure}

\subsection{The $A_{\infty}$ homomorphism} \label{sec:a_infty-homomorphism}
Given an element  $u \in \Pilbar(q, \vx, q)$, we obtain a path with endpoints on $q$ by considering the image of the outgoing segment starting at $\xi^0$ and ending on $\xi^{-1}$.  There is of course an ambiguity of parametrisation since the group of self-homeomorphisms of an interval acts on this space.  Using the parametrisation by arc length, we may compatibly eliminate this ambiguity:
\begin{lem}
There exists a choice of parametrisations of the outgoing boundary segment of half discs which yields maps
\begin{equation} \Pilbar(q, \vx, q) \to \Omega(q) \end{equation}
such that whenever $\vx[1]$, $\vx[2]$ and $x$ are as in Equation \eqref{eq:moduli_space_maps_breaking_top} we have a commutative diagram
\begin{equation*}
 \xymatrix{  \Pilbar (q, \vx[1], q ) \times  \Discbar(x;\vx[2]) \ar[r] \ar[d] &  \Pilbar(q, \vx, q)  \ar[d]  \\
 \Pilbar (q, \vx[1], q) \ar[r] &   \Omega(q) } \end{equation*}
while whenever  $\vx[1]$ and $\vx[2]$ are as in Equation \eqref{eq:moduli_space_maps_breaking_side}, the following diagram also commutes
\begin{equation*}
 \xymatrix{     \Pilbar( q, \vx[1], q) \times \Pilbar( q, \vx[2], q) \ar[r] \ar[d] & \Pilbar(q, \vx, q)  \ar[d] \\
   \Omega(q) \times  \Omega(q) \ar[r]  & \Omega(q) . }
\end{equation*} \noproof
\end{lem}

In particular, we obtain an evaluation map
\begin{equation*}
   C_{*}(\Pilbar(q, \vx, q)  )  \to  C_{*}(\Omega(q)).
\end{equation*}
According to Lemma \ref{lem:moduli_half_discs_manifold_corners}, the moduli spaces $ \Pilbar(q, \vx, q)  $ are manifolds with corners; by the construction explained in Appendix \ref{sec:signs-orientations}, we have a canonical up to homotopy isomorphism
\begin{equation} \label{eq:orientation_bundles_half-discs}
\lambda(\Pil(q,\vx,q)) \cong \lambda(\Pil_{d})  \otimes o_{q} \otimes o_{x_1}^{-1} \otimes \cdots \otimes  o_{x_d}^{-1} \otimes o_{q}^{-1}.
\end{equation}
In particular, these manifolds are orientable and hence admit a fundamental chain whose boundary represents $\partial  \Pilbar(q, \vx, q) $ once we fix orientations of $o_{x}$ for all chords.  The next result is a restatement of Lemma 4.14 of \cite{string-top}, with the signs verified in Appendix A of  \cite{string-top}.
\begin{lem} \label{lem:good_choice_fundamental}
There exists a family of fundamental chains 
\begin{equation} [\Pilbar(q, \vx, q)] \in C_{*}(\Pilbar(q, \vx, q)  )  \end{equation}
in the cubical chain complex whose boundary is given by
\begin{equation}  \label{eq:boundary_fundamental_chain}
  \sum_{\vx[1] \cup \vx[2]= \vx} (-1)^{\flat }   [\Pilbar( q, \vx[1], q)] \times  [\Pilbar( q, \vx[2], q)]  + \sum_{\vx[1] - \{x\} \cup \vx[2] = \vx } (-1)^{\sharp} [\Pilbar(q, \vx[1], q)]  \times [\Discbar(x;\vx[2])] 
\end{equation}
where the first sign is given by
\begin{equation}  \label{eq:sign_flat} \flat = (d_2+1)\left( \sum_{i=1}^{d_1} |x_{i}| \right) + d_1 +1 , \end{equation}
and the second sign is 
\begin{equation} \label{eq:sign_sharp} \sharp=  d_2 \left( \sum_{j=1}^{k+d_2} |x_j|  \right) + d_{2}(d-k) + k +1   \end{equation}
whenever $\Discbar(x;\vx[2]) $  is rigid  and  $x$ is the $k+1$\st element of $\vx[1]$. \noproof
\end{lem}

We now define a map
\begin{align}
\cF^{d} \co (CW^{*}_{b}(L))^{\otimes d} & \to C_{*}(\Omega(q)) \\
x_d \otimes \cdots \otimes x_1 & \to \bigoplus_{q} (-1)^{\dagger+ d |\vx|} \ev_{*} ( [\Pilbar(q, \vx, q)]) , \label{eq:twist_A_infty_homomorphism}
\end{align}
where $|\vx|$ is the sum of the degrees of the inputs.
\begin{lem}[Lemma 4.15 of \cite{string-top}]
The collection of maps $\cF^{d}$ satisfy the $A_{\infty}$ equation for functors 
\begin{multline} \label{eq:a_infty_functor_equation} \sum_{d_1+d_2 = d+1} (-1)^{\maltese_{1}^{i}} \cF^{d}\left( x_{d}, \ldots, x_{i+d_2+1}, \mu^{\F}_{d_2}(x_{i+d_2}, \ldots, x_{i+1}) , x_{i},  \ldots, x_1 \right) = \\  \mu_{1}^{\P} \left( \cF^{d}( x_{d}, \ldots, x_1) \right) + \sum_{d_1 + d_2 = d} \mu_{2}^{\P}\left(\cF^{d_2}(x_{d}, \ldots, x_{d_1+1})  , \cF^{d_1}(  x_{d_1},  \ldots, x_1 ) \right),   \end{multline}
where the sign on the left hand side is given by
\begin{equation*}
  \maltese_{1}^{i} = i + \sum_{j = 1}^{i} |x_i|. 
\end{equation*}
\noproof
\end{lem}

\section{The closed sector}
 \subsection{Construction of  (twisted) symplectic cohomology and the PSS homomorphism} \label{sec:constr-twist-sympl}
Let $F \co S^1 \times \TQ \to \bR$ be a smooth non-negative function such that
\begin{equation}
  \label{eq:support_condition_perturbation}
  \parbox{36em}{$F$ and  $  \lambda(X_F)$ are uniformly bounded in absolute value, and there is a sequence $R_i \to  +\infty$ such that $F(t, p,q) $ vanishes if $|p|$ lies in some open neighbourhood of $R_i$.}
\end{equation}

We write $ H_{S^1}$ for the sum of $H$ and $F$,   $X_{S^1}$ for the time-dependent Hamiltonian vector field of  $H_{S^1}$, and $\Orbit$ for the set of time-$1$ periodic orbits.  For a generic choice of $F$, all time-$1$ periodic orbits  of $ X_{S_1} $ are non-degenerate, and we define the degree of such an orbit in terms of the Conley-Zehnder index as
\begin{equation}
  \label{eq:formula_degree_CZ}
|y| = n - CZ(y) .
\end{equation}

To define the twisting of symplectic cohomology that corresponds to the fact that we shall work with the Fukaya category with respect to a non-trivial background class, consider the pullback of $E_{b}$ under an orbit.  This vector bundle over the circle admits two $\Spin$ structures because the space of such is an affine space over first cohomology with $\bZ_{2}$ coefficients, which has rank $1$.  Neither of these is preferred:
\begin{defin}
  The \emph{background line} $\kappa_{y}^{b}$ is the free abelian group generated by the two $\Spin$ structures on $y^{*}(E_b) $ with the relation that their sum vanishes. 
\end{defin}
\begin{rem}
Recall that the orientation line of a vector space is the free abelian vector space generated by its two orientations, with the relation that their sum vanishes.  Our background line is modeled after this more familiar notion, exploiting the fact that $\pi_{1}(O(n))$ is $\bZ_{2}$ whenever $n>2$.   Note that the definition makes sense for any class $b \in H^{2}(\TQ, \bZ_{2})$, not necessarily agreeing with the pullback of $w_{2}(Q)$.  Starting with a vector bundle on a manifold, this construction can be used to produce a canonical local system on the free loop space whose associated class in first cohomology is obtained by transgressing the second Stiefel-Whitney class of the bundle on the base.
\end{rem}

Given an $S^1$-dependent family  $I_{S^1} \in  \sJ(\TQ)$, we  write $\Cyl(y_0;y_1)$ for the quotient by $\bR$ of the moduli space of maps
\begin{equation*} u \co C \equiv (-\infty, +\infty) \times S^1 \to \TQ  \end{equation*}
converging exponentially at $+\infty$ to $y_0$ and at $-\infty$ to $y_1$, and  satisfying  Floer's equation
\begin{equation}
  \label{eq:dbar_cylinder}
 \left(du - X_{S^1}  \otimes  dt \right)^{0,1} = 0
\end{equation}
with respect to the $S^1$-dependent almost complex structure $I_{S^1} $.  The index theorem implies that $\Cyl(y_0;y_1)$ is $0$-dimensional whenever $|y_0| = |y_1| +1$, and that there are real lines $o_y$ (see Appendix C of \cite{generation}) associated to each periodic orbit such that every element of this moduli space induces, up to homotopy, a canonical isomorphism from $o_{y_1}$ to $o_{y_0}$ and hence a map on orientation lines.
\begin{equation}\label{eq:map_orientation_lines}
  |o_{y_1}| \to |o_{y_0}|.
\end{equation}
In brief, $o_{y}$ is the determinant bundle of a Cauchy-Riemann operator on $\bC$ whose asymptotic conditions at infinity are given by the linearisation of Equation \eqref{eq:dbar_cylinder} in a trivialisation of $y^{*}(\TQ)$ determined up to homotopy by the choice of a (complex) volume form on $\TQ$ fixed in Section \ref{sec:geom-prel}.

The exponential convergence of $u(s,t)$ to $y_0$ at $s=+\infty$ and $y_1$  at $s=-\infty$ implies that a $\Spin$ structure on  the pullback of $E_{b}$ under $y_1$ or $y_0$ induces one on the pullback of $E_{b}$ under $u$.  In particular, we also obtain a canonical isomorphism
\begin{equation}
  \label{eq:backround_line_twist}
\kappa_{y_1} \to \kappa_{y_0}.
\end{equation}

Writing $\partial_u$ for the tensor products of the maps in Equation \eqref{eq:map_orientation_lines} and \eqref{eq:backround_line_twist}, we define the symplectic chain complex 
\begin{align}
 \label{eq:symplectic_complex}
 SC^{i}_{b}(\TQ) & \equiv \bigoplus _{\substack{y \in \Orbit \\ |y| = i}} | o_y| \otimes \kappa_{y} \\
\partial([y_1]) & =  (-1)^{|y_1|} \sum_{u} \partial_u([y_1] ). \notag
\end{align}
The finiteness of the right hand side follows from Gromov compactness and a version of the maximum principle, and  the cohomology of this complex is called \emph{(twisted) symplectic cohomology} and denoted $SH^{*}_{b}(\TQ)$.  Since we only work with one twist of symplectic cohomology in this paper, we shall often refer to this simply as symplectic cohomology.

We shall now construct a map from the cohomology of $Q$ to symplectic cohomology.  There are alternative definitions of symplectic cohomology in which the complex is built from two parts, one generated by Reeb chords (or by Hamiltonian chords occurring away from a compact set), the other by critical points of a Morse function on $\TQ$.  From such a point of view, the existence of this map is obvious; we shall nonetheless avoid it because it would complicate the construction of various homomorphisms in and out of symplectic cohomology which will be used throughout the paper.  The presence of generators of different flavours would require a case-by-case analysis.

Instead, we shall use the work of  Piunikhin, Salamon, and Schwarz which constructs a chain equivalence between the Morse complex and the Floer chain complex in the case of compact manifolds, see \cite{PSS}.  Adapting their idea to our setting, we obtain a map
\begin{equation} \label{eq:PSS-map}
  H^*(\TQ) \to SH^{*}_{b}(\TQ)
\end{equation}
as follows:

Choose a $1$-form $\beta$ on $C=(-\infty, +\infty) \times S^1$ satisfying $d \beta \leq 0$ everywhere, which agrees with $dt$ near $-\infty$, and which vanishes near $+\infty$, as well as a family of almost complex structures $I_{PSS} \co C \to \sJ(\TQ)$ which agree with $I_t$ near $-\infty$, and which are independent of the source near $+\infty$.  These data allow us to impose the equation
  \begin{equation} \label{eq:PSS-equation}
    (du - X_{S^1} \otimes \beta)^{0,1}=0
  \end{equation}
on maps from the cylinder to $\TQ$.  Equation \eqref{eq:PSS-equation} reduces to the ordinary holomorphic curve equation for a constant almost complex structure near the positive end by our assumptions on $\beta$ and  $I_{PSS}$.  Finiteness of energy then implies that the map extends by adding a point at $+\infty$ (this is the removal of singularities theorem which in this case goes back to Gromov \cite{gromov}).  At $-\infty$, we obtain convergence to an orbit by a standard result in Floer theory.

Given a manifold $N$ with boundary equipped with a map to $ \TQ $ and an orbit of $H_{S^1}  $  we define $\Cyl(y,N) $  to be the space of solutions to Equation \eqref{eq:PSS-equation} which converge to $y $  at $-\infty$, and to a point in the image of $N$ at $+\infty$.  For a generic choice of $I_{PSS}$, this is a smooth manifold of dimension $\codim(N) - |y| $.   The Gromov bordification $\Cylbar(y,N) $ of this moduli space has two types of codimension $1$ strata:
\begin{equation*}
  \Cyl(y,\partial N) \cup \coprod_{|y_1|= |y|-1} \Cyl(y,y_1)   \times \Cyl(y_1,N)
\end{equation*}
corresponding to the point at $+\infty$ escaping to $\partial N$, and to the breaking of solutions to Floer's equation at $-\infty$.

The proof of compactness in Lemma \ref{lem:moduli_half_discs_manifold_corners} applies to this setting as well:
\begin{lem}
  If the map $N \to \TQ$ is proper, then   $\Cylbar(y,N) $  is compact. \noproof
\end{lem}

In particular, using the reader's favourite chain model for relative homology (the PSS isomorphism being usually phrased in terms of Morse chains), we obtain a chain map
\begin{equation*}
  H_{2n-*}(\TQ, \SQ) \to SH^{*}_{b}(\TQ)
\end{equation*}
by an appropriate count of those elements of $ \Cylbar(y,N) $ which are rigid; the PSS map in Equation \eqref{eq:PSS-map} is obtained by identifying the source with cohomology using Poincar\'e duality for manifolds with boundary.  More precisely, the PSS map defined, say in \cite{PSS} takes value in the untwisted version of symplectic cohomology.  To define it in the presence of a background class, we must be able to assign to every component of $ \Cylbar(y,N)  $ a trivialisation of $ \kappa_{y}^{b}$; i.e. a $\Spin$ structure on $y^{*}(E_{b})$.  

In order to do this, we exploit again the fact that any element $u \in  \Cylbar(y,N) $ extends continuously to a map
\begin{equation}
  \bar{u} \co \bC \to \TQ
\end{equation}
whose source the plane obtained by adding a point to the cylinder at $s=- \infty$.  We choose the $\Spin$ structure on $ y^{*}(E_{b})$ to be the unique one which extends to a $\Spin$ structure on $\bar{u}^{*}(E_{b})$.

\begin{example}
If we perturb the Hamiltonian $|p|^2$ on $T^* S^1$ by a $C^2$ small function then each Reeb orbit contributes two generators to   $SC^{i}(T^* S^1)$  in degrees $0$ and $1$ (there is no twist in this case so we drop $b$ from the notation).  As the Reeb orbits are in non-trivial homology classes, they cannot be in the image of the PSS homomorphism.  If we choose the perturbation to be autonomous (time-independent) in a neighbourhood of the zero section, then the critical points of the perturbed Hamiltonian give rise to generators in addition to the ones coming from Reeb chords; we may choose this perturbation so that there are two additional generator, again in degree $0$ and $1$  and the subspace generated by these is the image of the PSS homomorphism.  
\end{example}
\subsection{Moduli spaces of half-cylinders} \label{sec:count-half-cylind}
We define the space of Moore loops on $Q$ to be
\begin{equation*}
 \sL Q \equiv  \Map(\bR/\bZ, Q)  \times [0,+\infty).
\end{equation*}
In particular, projection to the second factor defines a continuous map
\begin{equation*}
  L \co \sL Q \to [0,+\infty) 
\end{equation*}
which we think of as recording the length of every loop.   When convenient, we shall parametrise a loop of length $L$ by the interval $[0,L]$ rather than $[0,1]$.

In this section, we define a chain map
\begin{equation}
  \label{eq:H(CL)}
H^{*}(\CL) \co  SH^{*}_{b}(\TQ) \to H_{n-*}(\sL Q)
\end{equation}
which counts half cylinders with boundary on $Q$.  The readers familiar with symplectic field theory should recognise that we are simply recasting the construction of Cieliebak and Latschev \cite{CL} in the language of Floer theory, which allows us to avoid the technical difficulties inherent to SFT.  One may also compare the construction we are about to give with that of the map $\CO$ in Section 5.2 of \cite{generation}.

We write $C^+$ for the positive half $ [0,+\infty) \times S^1$ of the cylinder with coordinates $(s,t)$, and pick maps  $I_{C^+} \co C^+ \to \sJ(\TQ)$ and $H_{C^+}  \co C^+ \to \sH(\TQ)$ which near infinity depend only on the $t$ variable, and agree respectively with $I_{t}$ and $H$.  Moreover, we require that, in a neighbourhood of the boundary of $C^+$, the restriction of $H_{C^+}$ to a neighbourhood of the zero section agree with $-F$.

Given a time-$1$ orbit $y_1$ of $X_{S^1}$, we define $\Disc^{1}(y_1)$ to be the space of finite energy maps 
\begin{equation*}
  u \co C^+ \to \TQ
\end{equation*}
with boundary and asymptotic conditions
\begin{equation}
  \label{eq:boundary_disc_1_interior_1_boundary_output}
  \begin{cases}  u(0,t)   \in Q &   \\
\lim_{s \to + \infty} u (s, \cdot) = y_1(\cdot) & 
 \end{cases}
\end{equation}
and solving the differential equation
\begin{equation}
  \label{eq:dbar_map_from_SH}
  \left( du -   (X_{H_{C^+}} + X_{F})   \otimes dt \right)^{0,1} = 0.
\end{equation}
The key point is that this equation agrees with Equation \eqref{eq:dbar_cylinder} at infinity, and with the usual $\dbar$ equation near the boundary since we have required the boundary to map to $Q$.  In particular, the codimension $1$ boundary strata of the Gromov compactification of  $\Disc^{1}(y_1)$ are  the images of the natural inclusions
\begin{equation}
  \label{eq:boundary_disc_from_SH}
 \Disc^{1}(y_0) \times \Cyl(y_0;y_1)   \to \partial  \Discbar^{1}(y_1).
\end{equation}

Choosing the data $H_{C^+}$ and  $I_{C^+}$ generically, we ensure that $ \Discbar^{1}(y_1)$ is a manifold with boundary of dimension of  $n -|y_1| $.  Applying the usual strategy for orienting moduli spaces of holomorphic curves (see Appendix \ref{sec:signs-orientations}), we find that a choice of relative $\Spin$ structure on $Q$ induces a canonical isomorphism
\begin{equation} \label{eq:iso_tangent_half_cylinder}
  \lambda(  \Discbar^{1}(y_1))  \cong \lambda(Q) \otimes \left( o_{y_1} \otimes \kappa_{y_1}^{b}\right)^{-1},
\end{equation}
i.e. determines an orientation of $  \Discbar^{1}(y_1) $ relative to the tangent space of $Q$ at the image of the basepoint $(0,0)$ and $o_{y_1}  \otimes \kappa_{y_1}^{b}$.  The boundary stratum \eqref{eq:boundary_disc_from_SH} inherits an orientation which we must compare with the product orientation coming from the isomorphisms
\begin{align*}
  \lambda(  \Discbar^{1}(y_0))  & \cong \lambda(Q) \otimes \left( o_{y_0} \otimes \kappa_{y_1}^{b} \right)^{-1} \\
 \langle \partial_{s} \rangle \otimes \lambda(  \Cyl(y_0;y_1) ) & \cong o_{y_0} \otimes \left( o_{y_1} \otimes  \kappa_{y_1}^{b}\right)^{-1} ,
\end{align*}
where $ \langle  \partial_{s}\rangle$ is the vector space of translations of the cylinder.  Taking the tensor product of these two isomorphisms, we have
\begin{align*}
  \lambda(  \Discbar^{1}(y_0))  \otimes \langle \partial_{s} \rangle \otimes \lambda(  \Cyl(y_0;y_1) ) \cong \lambda(Q)  \otimes  \left( o_{y_1} \otimes  \kappa_{y_1}^{b} \right)^{-1}.
\end{align*}
Since translating the cylinder in the direction of $ \partial_{s} $ moves every point away from $y_0$, it corresponds to an outward pointing normal vector.  Keeping track of the Koszul sign arising from permuting this line past $  \lambda(  \Discbar^{1}(y_0))  $, we find that there is a difference of 
\begin{equation*}
  n+|y_0|
\end{equation*}
between the two orientations.  In particular, we may choose fundamental chains for $ \Discbar^{1}(y_1)$ and $\Cylbar(y_0;y_1)     $ in cubical homology such that
\begin{equation}
  \label{eq:boundary_chains_discs_SH}
 \partial [\Discbar^{1}(y_1)]  =  \sum_{y_0} (-1)^{n+ |y_0|} [\Discbar^{1}(y_0)] \times [\Cylbar(y_0;y_1)].
\end{equation}

Restricting an element of $ \Discbar^{1}(y_1) $ to the boundary, we obtain a Moore loop on $Q$ whose base point is $1 \in S^1 = \partial D^2$.    Writing $\ev_{*}$ for the induced map on chains, we define 
\begin{align}
\label{eq:CL_chains}
\CL \co SC^{*}_{b}(\TQ) & \to \Lef_{n-*}(\sL Q) \\
\CL([y_1]) & = (-1)^{|y_1|}\ev_{*}( [\Discbar^{1}(y_1) ] ),
\end{align}
where the symbol $\Lef$  stands for the chain complex computing ordinary homology that we construct in Appendix \ref{sec:cuboidal-chains} as a quotient of the normalised cubical chain complex.

\begin{lem}
  $\CL$ is a degree $n$ chain map.
\end{lem}
\begin{proof}
First note that, while the boundary of $ \Discbar^{1}(y_1) $ contains strata where the factors have arbitrary dimension, only those for which the cylinder is rigid survive after evaluation into $ \Lef_{n-*}(\sL Q) $; this is a consequence of taking the quotient by degenerate chains.  Applying the evaluation map to Equation  \eqref{eq:symplectic_complex} we find 
\begin{align*}
\partial \CL([y_1]) & = (-1)^{|y_1|} \sum_{|y_0| = |y_1| +1}  (-1)^{n+ |y_0|}  \ev_{*}( [\Discbar^{1}(y_0 )]  \times  [\Cylbar(y_0;y_1)]) \\
 & = (-1)^{n} \CL ( \partial [y_1]).
\end{align*}
In the last step, we have incorporated both the sign difference $(-1)^{|y_0|}$ between the evaluation map and $\CL$, but also $(-1)^{|y_1|}$ coming from Equation \eqref{eq:symplectic_complex}.
\end{proof}

\subsection{The PSS homomorphism and constant loops}
In this section, we prove the following result:
\begin{lem} \label{lem:constants_factor}
The map $H^{*}(\CL)$ fits into a commutative diagram
\begin{equation} \label{eq:diagram_PSS}
  \xymatrix{  H_{2n-*}(\TQ , \SQ) \ar[r]^{PSS} \ar[d]^{\cong} &  SH^{*}_{b}(\TQ) \ar[d]^{H^*(\CL)} \\
H_{n-*}(Q)  \ar[r] & H_{n-*}(\sL Q)   }
\end{equation}
\end{lem}
In the above diagram, the bottom horizontal arrow is induced by the inclusion of constant loops, and the vertical arrow on the left may be expressed using Poincar\'e duality as a composition
\begin{equation*}
   H_{2n-*}(\TQ , \SQ) \cong H^{*}(\TQ) \cong H^{*}(Q)  \cong H_{n-*}(Q) .
\end{equation*}

 In order to prove this result, we consider a family $\beta^{r}$ of $1$-forms on $C^{+}$ parametrised by $r \in [0,\infty)$ each satisfying $d \beta^{r} \leq 0$,  and such that
\begin{equation}
 \parbox{36em}{$ \beta^{0} \equiv 0$ and whenever $r$ is sufficiently large, $\beta^{r}(s,t) = \beta(s-r,t)$.}
\end{equation}
Note that the second condition means that near $r=\infty$, $\beta^{r}$ is obtained by gluing the $1$-form $dt$ on the positive half cylinder to $\beta$, in particular, as $r$ grows, $\beta^{r}$ agrees with $dt$ in an expanding neighbourhood of the boundary, and the support of $d \beta$ is pushed to $s=+\infty$.  

Let us in addition choose a map $[0,\infty) \times C^+ \to \sJ(\TQ)$ which similarly agrees with $I_{C+}$ whenever $r=0$ and is obtained by gluing $I_{C^+}$ and $I_{PSS}$ whenever $r$ is sufficiently close to $+\infty$.  With this data, we consider the equation
\begin{equation} \label{eq:homotopy_PSS}
  (du -  (X_{H_{C^+}} + X_{F})  \otimes \beta^{r})^{0,1} = 0, 
\end{equation}
and  define, for each submanifold $N$ of $\TQ $ the moduli space  $\Disc^{1}_{[0,\infty)}(N)$ to be the union of the spaces of solutions $u \co C^+ \to \TQ$ to this Cauchy-Riemann equation for some $r \in [0,\infty)$ which map $\partial C^+$ to $Q$ and converge to a point in $N$ at $+\infty$.

\begin{proof}[Sketch of the proof of Lemma \ref{lem:constants_factor}]
The Gromov bordification of  $\Disc^{1}_{[0,\infty)}(N)  $ is compact because all almost complex structures are of contact type near the boundary.  Note that whenever $r=0$, all solutions to Equation \eqref{eq:homotopy_PSS} are constant.  The standard transversality package therefore implies that, as long as the almost complex structure is chosen generically and $N$ meets $Q$ transversely, $ \Disc^{1}_{[0,\infty)}(N) $ is a smooth manifold  of dimension $\codim(N) +1$ with boundary $Q \cap N$.  

Using the length parametrisation and restricting every map to the boundary of $C^+$, we obtain an evaluation map 
\begin{equation*}
  \Discbar^{1}_{[0,\infty)}(N) \to \sL Q
\end{equation*}
from the Gromov compactification.  We claim that the image of a fundamental chain on this moduli space defines the chain-level homotopy which establishes the commutativity of Diagram   \eqref{eq:diagram_PSS}.  Note that image, under the evaluation map, of the total boundary of  this moduli space's fundamental chain corresponds to composing the homotopy with the differential in $C_{n-*}(\sL Q)$; the desired result shall follow by interpreting different boundary strata to account for the remaining terms in the equation for a homotopy.

Since taking the intersection of $N$ with $Q$ represents on homology the result of applying the homomorphism  $H_{2n-*}(\TQ , \SQ)  \to H_{n-*}(Q) $ to the fundamental cycle of $N$,  the stratum of $\partial \Discbar^{1}_{[0,\infty)}(N)  $ corresponding to $r=0$ represents the inclusion of constant loops. 

 By  letting the parameter $r$ go to $\infty$, we obtain the stratum
\begin{equation*}
 \Discbar^{1}(y) \times \Cylbar(y,N) 
\end{equation*}
which may be interpreted algebraically as the composition of $PSS$ with $\CL$.    The remaining part of the boundary is covered by the image of $  \Disc^{1}_{[0,\infty)}(\partial N) $ which corresponds to applying the differential and then the homotopy.
\end{proof}
\begin{rem}
One has several options in order to realise the map $N \to Q \cap N$ as a chain map inducing the homomorphism  $H_{2n-*}(\TQ , \SQ)  \to H_{n-*}(Q) $.  If one works with cubical chains, one may consider the subcomplex of \emph{locally finite} cubical chains generated by maps whose restriction to every stratum is transverse to $Q$.  For each such chain, the intersection with $Q$ is a manifold with corners for which we may choose fundamental chains in cubical homology by induction.  There are alternative models using Morse chains.
\end{rem}

\section{From the open to the closed sector}

\subsection{The bar model for Hochschild homology} \label{sec:bar-model-hochschild}
Given an $A_{\infty}$ algebra $\cA$, consider the graded vector space
\begin{equation*}
   CC_{*}^{(d)}(\cA) = \cA \otimes \left(  \cA[1] \right)^{\otimes d-1}.
\end{equation*}
 The cyclic bar complex of $\cA$ is the direct sum
 \begin{equation}
   \label{eq:cyclic_bar_complex}
    CC_{*}(\cA) = \bigoplus_{d}  CC_{*}^{(d)}(\cA)
 \end{equation}
equipped with the Hochschild differential
\begin{multline}
  b (a_{d} \otimes \cdots \otimes a_{1}) =  \sum_{1 \leq  i+j < d} (-1)^{\maltese_{1}^{i}} a_{d} \otimes \cdots \otimes a_{i+j+1} \otimes \mu_{j}(a_{i+j} , \ldots, a_{i+1}) \otimes a_{i} \otimes \ldots \otimes a_{1}  \\
+  \sum_{ 0 \leq i+j < d} (-1)^{  \shuffle_{i}^{i+j} + \maltese_{i+1}^{i+j} +1} \mu_{d-j-1}( a_{i}, \ldots, a_1, a_d, \ldots, a_{i+j+1} ) \otimes a_{i+j} \otimes \ldots \otimes a_{i+1}
  \label{eq:hochschild_differential}
\end{multline}
where the second sign, using the convention that $\maltese_{*}^{**}$ stands for the sum of the reduced degrees ($| \cdot | +1$) of elements between $*$ and $**$, may be expressed as
\begin{equation} \label{eq:sign_hochschild}
 \shuffle_{i}^{i+j} = \maltese_{1}^{i} \cdot (1+ \maltese_{i+1}^{d}) + \maltese_{i+j+1}^{d-1}.
\end{equation}
\begin{rem}
This is  a good opportunity to give some heuristics about signs: Our convention is that $ CC_{*}(\cA) $  is cohomologically graded (despite its name), with the degree of $  a_{d} \otimes \cdots \otimes a_{1}$ given by the sum of the degrees of each factor with the understanding that $a_{d}$ is given its usual degree, with every other letter  assigned its reduced degree.  Since all operations $\mu_{d}$ are of degree one with respect to the reduced degree, we introduce a sign $|a|+1$ whenever we permute an operation ``past'' an element $a$ of the algebra (except when it is in last position).

To obtain the sign in Equation \eqref{eq:hochschild_differential} from these considerations, we must first agree that the operations $\mu_{d}$ are applied from the right.  With this in mind, we obtain a sum
\begin{equation*}
  \maltese_{1}^{i} \cdot (1 + \maltese_{i+1}^{d}) + \maltese_{i+1}^{i+j} + \maltese_{i+j+1}^{d-1}+1
\end{equation*}
where the first expression comes from permuting $ a_{i}\otimes \cdots \otimes a_1 $ past the other terms, the second from permuting $\mu_{d-j-1}$ past $ a_{i+j} \otimes \ldots \otimes a_{i+1} $, and the third from permuting an invisible symbol of degree $1$ from its position to the right of $a_d$ to a position just before $a_{i+j}$.  This invisible symbol records the fact that the first term in the bar complex, unlike all the others, is assigned its ordinary degree.
\end{rem}

Given an $A_{\infty}$ homomorphism $\cF \co \cA \to \cB$ with polynomial terms $\cF^{d}$, we have an induced map $CC_{*}(\cF)$ on Hochschild chains given by
\begin{multline}
 \label{eq:CC_map}
 a_{d} \otimes \cdots \otimes a_1  \mapsto \sum (-1)^{ \shuffle_{s_1}^{s_k} } \cF^{ d - s_1 + s_{k}}(a_{s_1},  \ldots, a_{1} , a_{d}, \ldots, a_{s_k+1})  \\ \otimes \cF^{s_k - s_{k-1}}(a_{s_k }, \ldots, a_{s_{k-1}+1} ) \otimes \ldots \otimes   \cF^{s_{2} - s_{1}}(a_{s_{2}}, \ldots, a_{s_1+1} ) .
\end{multline}
\begin{lem}
If $\cF$ is a quasi-isomorphism, then $CC_{*}(\cF)  $  induces an isomorphism on Hochschild homology. \noproof
\end{lem}

\subsection{An ad-hoc model for Goodwillie's map} \label{sec:an-ad-hoc}
 In this section, we construct a chain map
\begin{equation}
  \label{eq:goodwillie_map}
  \G \co CC_{*}(C_{-*}(\Omega_q Q)) \to \Lef_{-*}(\sL Q)
\end{equation}
where the left hand side is the cyclic bar complex of $ C_{-*}(\Omega_q Q) $, and the right hand side is a chain model for the homology of the free loop space using a quotient of cubical chains described in Appendix  \ref{sec:cuboidal-chains}. 

Let us write $\iota$ for the inclusion of $ \Omega_q Q $ in $\sL Q$.  We define
\begin{equation}
  \label{eq:G_map_words_1}
 \G^{(1)} \co CC^{(1)}_{*}( C_{-*} \Omega_{q} Q) \to \Lef_{-*}(\sL Q)
\end{equation}
on elements of degree $i$ to be the composition of $(-1)^{i}\iota_{*}$ with the projection map from $ C_{-*}(\sL Q) $ to $  \Lef_{-*}(\sL Q) $  which is a quasi-isomorphism by Corollary \ref{cor:standard_model_chains}.   We have to introduce the sign $(-1)^{i}$ because the differential on $  CC^{(1)}_{*}( C_{-*} \Omega_{q} Q)$ is the negative of the one on $ C_{-*} \Omega_{q} Q $.

For each  number $t \in [0,1]$, we now define a map
\begin{equation}
   \#_{t} \co  \Omega_{q} Q \times  \Omega_{q} Q \to \sL_{Q}.
\end{equation}
 The length of $\gamma_1  \#_{t} \gamma_{2} $ is the sum $\ell_1+\ell_2$ of the lengths of $\gamma_1$ and $\gamma_2$, and the parametrisation by the interval $[0,\ell_1 + \ell_2]$ is given by 
\begin{equation}
\gamma_{1}   \#_{t} \gamma_{2} \left(s  \right) = 
\begin{cases} 
\gamma_{1} \left( s+  t \ell_{1} \right) & \textrm{if } s \in  [0, (1-t) \ell_1]  \\
\gamma_{2}\left( s -  (1-t) \ell_1 \right)  &\textrm{if } s \in  [(1-t) \ell_1, (1-t) \ell_1 + \ell_2]   \\
\gamma_{1} \left( s -  \ell_2 -  (1-t) \ell_1  \right) & \textrm{otherwise} 
 \end{cases} 
\end{equation}
The idea is simply to concatenate $ \gamma_{1} $ and $ \gamma_{2} $ then use the parameter $t$ to move the base point of the loop ``around'' $\gamma_1$ so that  $ \gamma_1  \#_{t} \gamma_{2} $  agrees with  the usual concatenation of the two loops in the two different orders whenever  $t = 0,1$ (see Figure \ref{fig:loops_compose}).
\begin{figure}
  \centering
  \includegraphics{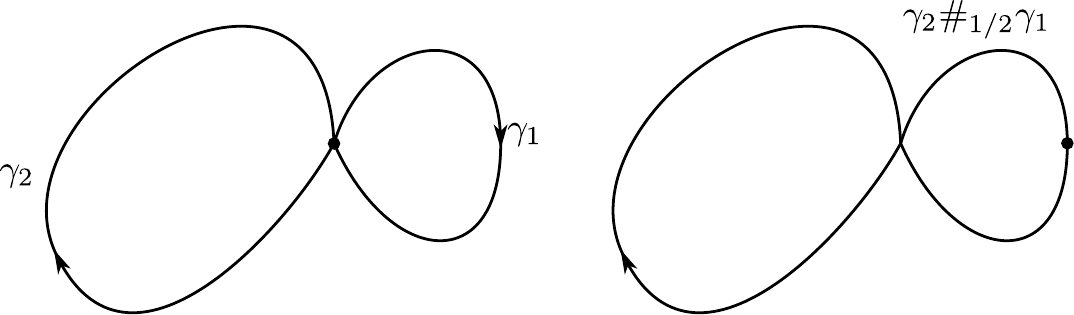}
  \caption{ }
  \label{fig:loops_compose}
\end{figure}

Given a pair of cubical chains $\tau$ and $\sigma$ of dimensions $i$ and $j$, with values in $ \Omega_{q} Q $, we define a cubical chain of dimension $i+j+1$ by taking the product of $\tau$ and $\sigma$, concatenating the corresponding loops, then using the last variable to ``move the basepoint'' around the loop coming from $\sigma$:
\begin{align} \label{eq:moving_basepoint}
\tau \btimes \sigma \co I^{i+j+1} & \to \sL Q  \\
(t_1, \ldots, t_{i+j+1})  & \mapsto  \tau(t_1, \ldots, t_i)  \#_{t_{i+j+1}} \sigma(t_{i+1}, \ldots, t_{i+j}) .
\end{align}

We now define the value of $\G$ on words of length $2$ as
\begin{align}
  \label{eq:G_map_words_2}
 \G^{(2)} \co CC^{(2)}_{*}( C_{-*} \Omega_{q} Q) & \to \Lef_{-*}(\sL Q) \\
\sigma_{2} \otimes \sigma_{1} & \mapsto - \sigma_{1}  \btimes \sigma_2,
\end{align}
and prescribe that it vanish on longer words.
\begin{lem}
$\G$  is a chain map.
\end{lem}
\begin{proof}
It is clear that the restriction of $\G$ to words  of length greater than $3$ commutes with the differential, while the case of length $1$ was discussed just after Equation \eqref{eq:G_map_words_1}.  For words of length $2$,  we compute that
\begin{align*}
  \partial  \G(\sigma_2 \otimes \sigma_1)  & =  -  \partial \left( \sigma_{1}  \btimes \sigma_2  \right) \\
& =  - \left( \partial \sigma_1\right)  \btimes \sigma_2  - (-1)^{|\sigma_1|}  \sigma_{1}  \btimes\left( \partial \sigma_2\right) + (-1)^{|\sigma_1| +|\sigma_2|}  \iota_{*}( \sigma_{1}   \#_{0} \sigma_{2} ) \\
& \qquad + (-1)^{|\sigma_1| +|\sigma_2|+1}  \iota_{*}( \sigma_{1}   \#_{1} \sigma_{2} ) \\ 
& =  \G( \sigma_{2} \otimes \partial \sigma_1 )  +  (-1)^{|\sigma_1|}   \G( \partial \sigma_{2} \otimes  \sigma_1 ) + \G(\sigma_{1} \cdot  \sigma_2)   + (-1)^{ |\sigma_2| \cdot |\sigma_1|+1} \G(\sigma_2 \cdot \sigma_1 ) .
\end{align*}
For the last term, observe  that $\sigma_{1}   \#_{1} \sigma_{2}  $ agrees with $ \sigma_2 \cdot \sigma_1  $ only after applying a permutation which identifies the  products  $[0,1]^{|\sigma_2| } \times [0,1]^{|\sigma_1| } $ and $[0,1]^{|\sigma_1| } \times [0,1]^{|\sigma_2| }$;  these cubical chains are identified in the chain complex  $\Lef_*$ up to the appropriate sign because we have taken the quotient by the subcomplex $\permute $ given in Equation \eqref{eq:subcomplex_permute}.

Writing out the last line in terms of the $A_{\infty}$ structure introduced in the beginning of Section \ref{sec:open-sector}, we find that
\begin{align*}
   \partial  \G(\sigma_2 \otimes \sigma_1)  & =  \G( \sigma_{2} \otimes \mu_{1}^{\P} \sigma_1 )  + (-1)^{|\sigma_1|}  \G( \mu_{1}^{\P}\sigma_{2} \otimes  \sigma_1 ) + (-1)^{|\sigma_1|} \G( \mu_{2}^{\P}(\sigma_{2},  \sigma_1) )  \\
& \qquad + (-1)^{|\sigma_2| + |\sigma_2| \cdot |\sigma_1|+1} \G(\mu_{2}^{\P}(\sigma_1, \sigma_2 )) \\
& = \G(b( \sigma_2 \otimes \sigma_1)).
\end{align*}

It remains therefore to show that given a word $ \sigma_{3} \otimes \sigma_2 \otimes \sigma_1 $, we have
\begin{multline*}
 \G( \sigma_{3}, \mu^{\P}_{2}(\sigma_2 \otimes \sigma_1))  + (-1)^{|\sigma_1|} \G( \mu^{\P}_{2}(\sigma_{3}, \sigma_2) \otimes \sigma_1)  \\  +  (-1)^{|\sigma_2| + (|\sigma_1| +1)(|\sigma_3| + |\sigma_2| +1)}  \G( \mu^{\P}_{2}(\sigma_{1}, \sigma_3) \otimes \sigma_2)  = 0 
\end{multline*}
This cancellation comes from taking the quotient of the usual cubical chains by the subcomplex $\fit$ defined in Equation \eqref{eq:chains_that_fit}.  Indeed, the cell $\sigma_{3} \btimes \mu_{2}^{\P}(\sigma_2 , \sigma_1) $ can be split into two cells which, up to permuting the coordinates, can be identified respectively with $ \mu_2^{\P}(\sigma_{3} , \sigma_2) \btimes \sigma_1 $  and  $ \mu_2^{\P}(\sigma_{1} , \sigma_3) \btimes \sigma_2$. 
\end{proof}

\begin{example}
Let $\gamma \co S^1 \to S^1$ denote the identity map, and $\gamma^{-1}$ the inverse loop.  Since concatenation of loops does not define a strictly commutative product on $  C_{-*} \Omega_{q} S^1 $, the Hochschild chain
\begin{equation*}
  \gamma^{-1} \otimes \gamma \in  CC^{(2)}_{-1}( C_{-*} \Omega_{q} S^1) 
\end{equation*}
is not closed, but the fact that it is homotopy commutative implies that there is a chain  $\sigma$ in $C_{1} \Omega_{q} S^1  $  which may be chosen among contractible loops such that $b( \gamma^{-1} \otimes \gamma ) = \partial \sigma $.  In particular,
\begin{equation*}
  \gamma^{-1} \otimes \gamma +  \sigma
\end{equation*}
represents a class in the Hochschild homology of $C_{-*} \Omega_{q} S^1$.

The image of this class under $H_{*}(\G)$ is the fundamental class of the circle, included in its free loop space as the space of constant loops.  The easiest way to see this is to recall that the base point projection map
\begin{equation*}
\pi \co \sL S^1 \to S^1
\end{equation*}
induces an isomorphism on homology when restricted to contractible loops.  Since $ \pi_{*} \circ \G (\sigma) $ is a degenerate chain, we simply observe that the base points of the family of loops $ \gamma^{-1} \#_{t} \gamma $ cover $S^1$ with multiplicity one as $t$ ranges between $0$ and $1$. 
\end{example}

\subsection{Review of the map $\OC$}
We shall now recall the definition of the map $\OC$ constructed in Section 5.3 of \cite{generation}.  First, we define $\Discbar_{d}^1$  to be the Deligne-Mumford compactification of the moduli space of holomorphic discs with $1$ interior puncture  and $d$ boundary punctures ordered counterclockwise (see Figure \ref{fig:disc^1_3}): we choose a cylindrical negative end at the interior puncture and positive strip-like ends at the boundary punctures which depend smoothly on the modulus and which near each stratum agree with the ends obtained by gluing.

\begin{figure}
  \centering
  \includegraphics{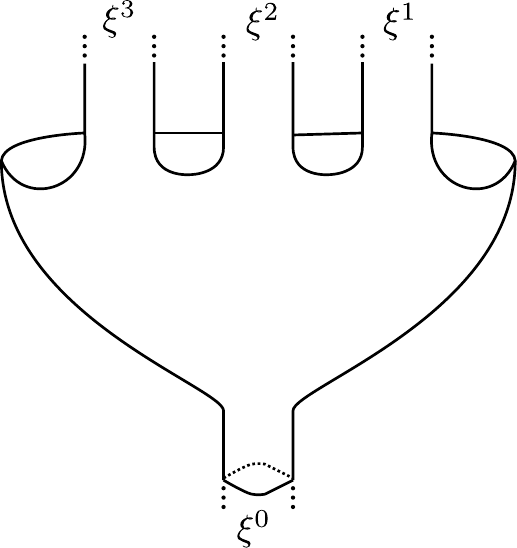}
  \caption{ }
  \label{fig:disc^1_3}
\end{figure}

\begin{defin} \label{def:floer_datum_discs_interior_puncture}
A \emph{Floer datum} $D_{S}$ on a stable disc $S \in \Discbar_{d}^{1}$ with $d$ positive boundary punctures $(\xi^1, \ldots, \xi^d) $ and one negative interior puncture $\sigma$ consists of the following choices on each component:
\begin{enumerate}
\item Time shifting map:  A map $ \rho_{S} \co \partial S \to [1,d]$ which is constant near each marked point.  We write $w_{k,S}$ for the value on the $k$\th end and set  
  \begin{equation}
    w_{0,S} = \sum_{k=1}^{d} w_{k,S}.
  \end{equation}
\item Basic $1$-form and Hamiltonian perturbations:  A closed $1$-form $\alpha_{S}$ whose restriction to the boundary vanishes and a map $H_{S} \co S \to \sH(\TQ)$  on each surface such that the pullback of $ X_{H_S} \otimes \alpha_{S} $ under the $k$\th end agrees with  $  X_{\frac{H}{w_{k,S}} \circ \psi^{w_{k,S}}} \otimes dt$.
\item Subclosed $1$-form: A $1$-form $\beta_{S}$ which may be written as the product of a smooth function with $\alpha_{S}$,  satisfying $d \beta_{S} \leq 0$,  and whose pullback under the $k$\th end vanishes unless $k=0$, in which case it agrees with $dt$.
\item Almost complex structures:  A map $I_{S} \co S \to \sJ(\TQ)$ whose pullback under the $k$\th end agrees with  $ (\psi^{w_{k,S}})^{*} I_{t} $ unless $k=0$ in which case it agrees with  $ (\psi^{w_{0,S}})^{*} I_{S^1}$. 
\end{enumerate}

A  \emph{universal and conformally consistent}  choice of Floer data for the map $\OC$ is a choice $\bfD_{\OC}$ of Floer data  for every integer $d \geq 1$, and every (representative) of an element of  $\Discbar_{d}^1$ which vary smoothly over the compactified moduli space, such that the two natural Floer data (coming from $\bfD_{\OC}  $ or $ \bfD_{\mu}$)  on any irreducible component of a singular disc are conformally equivalent, and which agree, to infinite order near each stratum,  with the Floer data obtained by gluing.
\end{defin} 

An inductive construction implies the existence of such universal Floer data in sufficient abundance to guarantee transversality.  In particular, given a sequence of chords  $\vx = \{ x_1, \ldots, x_d \}$ and an orbit $y_0 \in \Orbit$, we define $ \Disc_{d}^1(y_0; \vx) $ to be the moduli space of maps $u \co S \to \TQ$, with $S$ an arbitrary element of $ \Disc_{d}^1 $, such that $u(\partial S)$  lies in $\Tq$, which satisfies the appropriate  asymptotic conditions along the ends and solves the differential equation
\begin{equation}
  \label{eq:dbar_disc_1_interior_output}
\left(du - X_{H_S} \otimes \alpha_{S} - X_{\frac{F}{w_{0,S}} \circ \psi^{w_{0,S}}} \otimes \beta_{S} \right)^{0,1} = 0
\end{equation}
where the $(0,1)$ part is taken with respect to the $S$-dependent almost complex structure, and the function $F$ is the one appearing in the definition of symplectic cohomology.

Assuming that transversality is satisfied and that $|y_0| =  n -d + 1 + \sum_{1 \leq k \leq d} |x_k|  $, we conclude that the elements of  $\Disc_{d}^1(y_0; \vx) $ are rigid, and that we may canonically associate to each disc $u$ an isomorphism
\begin{equation}
  o_{x_d} \otimes \cdots \otimes o_{x_1} \to o_{y_0}
\end{equation}
using the conventions explained in Appendix \ref{sec:signs-orientations}.

Writing $\OC_{u}$ for the induced map on orientation lines, we define a map $\OC_d$
\begin{equation}
  \OC_d([x_d], \ldots, [x_1])  = \sum_{\stackrel{|y_0| =n -d + 1 + \sum_{1 \leq k \leq d} |x^k| }{u \in \Disc_{d}^{1}(y_0, \vx) }} (-1)^{|x_d| + \dagger} \OC_{u}( [x_d], \ldots, [x_1])
\end{equation}
where $\dagger$ is given in Equation \eqref{eq:dagger_sign}. 

As proved in Lemma 5.4 of \cite{generation}, these maps are the components of a degree $n$ chain map
\begin{equation}
\OC \co CC_{*}(CW^{*}_{b}(\Tq)) \to SC^{*}_{b}(\TQ)
\end{equation}
where the left hand-side is the cyclic bar complex of $CW^{*}_{b} (\Tq)$.

\section{Construction of the Homotopy}

In this section, we prove Proposition \ref{prop:commutative_diagram} by constructing a homotopy between the two possible chain level compositions.  We begin by introducing some abstract moduli spaces of holomorphic curves that will appear in the construction.

\subsection{Adding a marked point on the outgoing segment} \label{sec:adding-marked-point}
We shall consider the moduli space $\Pil_{d,1} $ of half-discs with $d$ incoming ends and one marked point on the outgoing segment.  By forgetting this marked point, we obtain a submersion to $\Pil_{d}$, with fibre an interval, which extends to the Gromov compactification
\begin{equation}
  \Pilbar_{d,1} \to \Pilbar_{d}.
\end{equation}
The fibre of this map over a point in the boundary is still topologically an interval, although over a stratum of codimension greater than $1$, such an interval will intersect several strata.  Indeed, when a half disc breaks into two half discs, then the outgoing segment breaks into two, and the basepoint can lie in either part of the outgoing interval as shown in Figure \ref{fig:half_discs_basepoint}.
\begin{figure}
  \centering
  \includegraphics{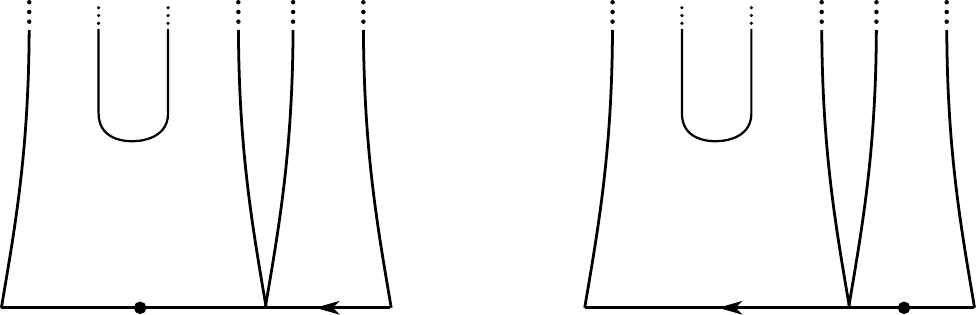}
  \caption{ }
  \label{fig:half_discs_basepoint}
\end{figure}

Given a sequence $\vx$ of chords with endpoints in $\Tq$, recall that we defined a moduli space $ \Pilbar_{d} (q,\vx,q)  $ in Section \ref{sec:moduli-space-half} of discs all of whose boundary segments map to $\Tq$, except the outgoing segment which is mapped to $Q$.  Let us consider the fibered product
\begin{equation*}
   \Pilbar_{d,1} (q,\vx,q) =   \Pilbar_{d} (q,\vx,q) \times_{\Pilbar_{d}}  \Pilbar_{d,1} .
\end{equation*}
Note that this definition makes sense if we chose Floer data on $  \Pilbar_{d,1}  $ coming from the forgetful map to $ \Pilbar_{d} $.   Moreover, using the length parametrisation of the outgoing segment, we shall fix an identification
\begin{equation} \label{eq:parametrise_half_disc_marked}
\Pilbar_{d} (q, \vx, q) \times [0,1] \to   \Pilbar_{d,1}(q, \vx, q).
\end{equation}
In particular, the product of a cubical chain in $ \Pilbar_{d}(q,\vx,q)   $ with $[0,1]$ defines a cubical chain in   $ \Pilbar_{d,1}(q,\vx,q) $, which gives us a preferred cubical fundamental chain for this moduli space.  Using the expression \eqref{eq:boundary_fundamental_chain} for the boundary of the fundamental chain of $\Pilbar_{d}(q,\vx,q)  $, we conclude that the fundamental chains of $  \Pilbar_{d,1}(q,\vx,q) $  satisfy the inductive relation:
\begin{multline} \label{eq:product_moduli_marked_point}
  \partial [ \Pilbar_{d,1}(q,\vx,q)]  = (-1)^{1+||\vx||} \left(   \iota_{1}( [\Pilbar_{d}(q,\vx,q)] ) - \iota_{0}( [\Pilbar_{d}(q,\vx,q)] ) \right) + \\  \sum_{\vx[1] \cup \vx[2]= \vx} (-1)^{\flat}  \left(  (-1)^{1+ ||\vx[2] ||}   [\Pilbar_{d_1,1}( q, \vx[1], q)] \times  [\Pilbar_{d_2}( q, \vx[2], q)]    +     [\Pilbar_{d_1}( q, \vx[1], q)] \times  [\Pilbar_{d_2,1}( q, \vx[2], q)] \right) \\ +  \sum_{x} (-1)^{\sharp+ \dim(\Discbar(y,\vx[2])  )  } [\Pilbar_{d_1,1}(q, \vx[1], q)]  \times [\Discbar(x,\vx[2])]    
\end{multline}
where $  \iota_{0} $ and $  \iota_{1} $  refer respectively to the inclusions at the two endpoints of the interval $[0,1]$, the rest of the inclusions are suppressed, and $|| \cdot ||$ is the sum of the reduced degrees  of the elements appearing in a sequence. By Lemma \ref{lem:moduli_half_discs_manifold_corners}, $1+|| \vx||$ is the dimension of the moduli space $\Pilbar_{d}(q,\vx,q)  $, which explains the appearance of the first sign as a Koszul sign arising from permuting the degree one operator $\partial$ with the fundamental chain.  In the second and third lines, $  ||\vx[2] || $ and $  \dim(\Discbar(y,\vx[2])  $ are introduced because the interval $[0,1]$ appears last in Equation \eqref{eq:parametrise_half_disc_marked}; at the relevant boundary stratum, we have to permute it past one of the factors in order to obtain the desired product decomposition.

 It is important to note that such a relation in general is satisfied in the quotient complex $\Lef_{*}(   \Pilbar_{d,1}(q,\vx,q) )$ but not necessarily in the usual cubical chain complex: by taking the product with $[0,1]$, a cubical chain in $ \Pilbar_{d_1}( q, \vx[1], q) \times  \Pilbar_{d_2}( q, \vx[2], q) $  produces a single chain supported on the boundary of  $  \Pilbar_{d,1}(q,\vx,q) $.    Because we take the quotient by the subcomplex \eqref{eq:chains_that_fit}, this chain is equal in $\Lef_{-*}(\sL Q)$ to the sum of the two chains coming from taking the product with $[0,1]$, and using the inclusion of the strata $ \Pilbar_{d_1}( q, \vx[1], q) \times  \Pilbar_{d_2,1}( q, \vx[2], q) $  and $  \Pilbar_{d_1,1}( q, \vx[1], q) \times  \Pilbar_{d_2}( q, \vx[2], q)$.

\subsection{An abstract moduli space of annuli}

\begin{figure}
  \centering
  \includegraphics{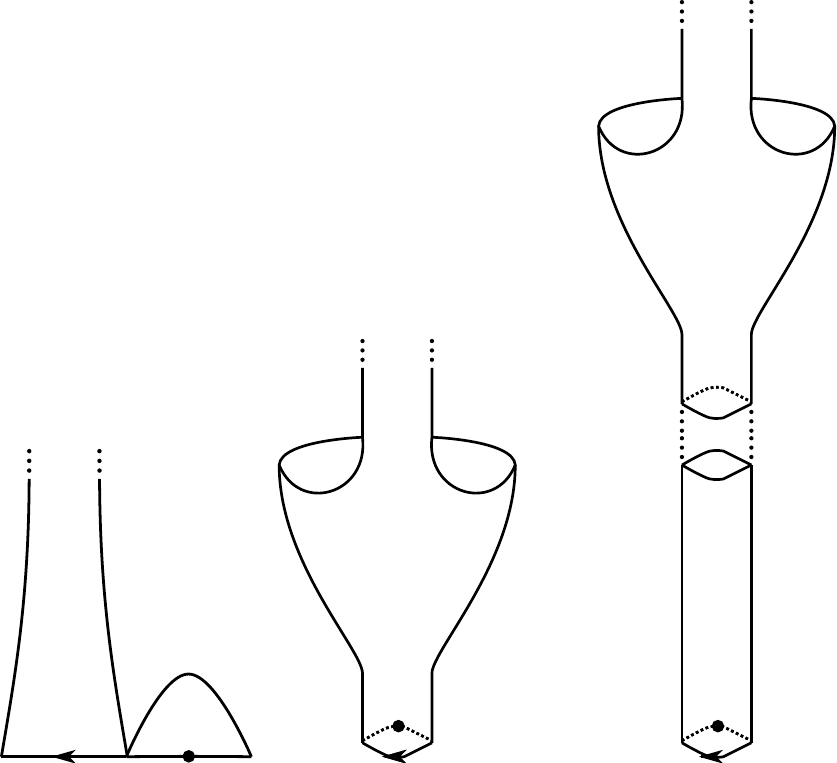}
  \caption{ }
  \label{fig:half_discs_1_input}
\end{figure}
 We write $\Ann^{-}_{1} $ for the moduli space of annuli with one marked point on a boundary component and an (incoming) puncture on the other and which are biholomorphic, for some positive real number $r$, to a domain 
\begin{equation}
  \label{eq:explicit_annulus}
  \{ z \in \bC | 1 \leq |z| \leq r \}
\end{equation}
with a puncture at $1$ and a marked point at $-r $.  We write $  \Ann^{-}_{d} $ for the space of annuli obtained by adding $d-1$ boundary punctures on the unit circle to one of the elements of $\Ann^{-}_{1} $: the resulting punctures are ordered counterclockwise ending with the one corresponding to $1$, and we call the boundary component carrying the marked point the \emph{outgoing circle}.  Writing $z_i$ for the coordinates which record the positions of the punctures on the circle, we fix the orientation
\begin{equation}
  \label{eq:orientation_moduli_annuli}
  dr \wedge dz_1 \wedge \cdots \wedge dz_{d-1}
\end{equation}
on $ \Ann^{-}_{d} $.

We shall compactify $ \Ann^{-}_{1} $ to a closed interval denoted $\Annbar^{-}_{1}$;  the outermost pictures in Figure \ref{fig:half_discs_1_input} show the broken curves representing the boundary.  More generally, the Deligne-Mumford compactification $\Annbar^{-}_{d}$ is a manifold with boundary whose codimension $1$ strata are the images of natural inclusions of the products
\begin{align}
\label{eq:boundary_moduli_annuli_interior_breaking}
\Discbar^{1} \times \Discbar_{d}^{1}    &   \\
\label{eq:boundary_moduli_annuli_coproduct_relation}
 \Pilbar_{d_1,1}  \times  \Pilbar_{d_2} &  \quad 1 \leq d_1 \leq d = d_1 +d_2 \\
 \label{eq:boundary_moduli_annuli_boundary_bubbling}
\Annbar^{-}_{d_1} \times \Discbar_{d_2}  &  \quad 1 \leq d_1 < d = d_1 + d_2 -1.
\end{align}
The first type of stratum arises from compactifying the end of the moduli space where the modular parameter $r$ reaches infinity, while the last type of stratum reflects the breaking of discs from the boundary component carrying the incoming marked points:  as such breaking could occur at any incoming point, there are $d_1$ distinct strata within the boundary of $ \Annbar^{-}_{1}  $ each being the image of the inclusion \eqref{eq:boundary_moduli_annuli_boundary_bubbling}.    The second type of stratum compactifies the end where the modular parameter converges to $0$.  The last incoming end and the marked point lie on different components of such a stable annulus since $\Annbar^{-}_{d}  $  consists of annuli for which these two points are ``opposite from each other.''  In particular, there are  $d_2$ different strata of the second type, distinguished by the position of the last incoming point of the annulus among the incoming points of $  \Pilbar_{d_2} $.  In Figure \ref{fig:half_discs_2_inputs}, we show generic elements of the four codimension $1$ strata for $d=2$.

As usual, we fix strip-like ends near the incoming ends of a stable annulus, which vary smoothly over the moduli space, and are compatible near a boundary stratum with the ones induced by gluing.
\begin{figure}
  \centering
  \includegraphics{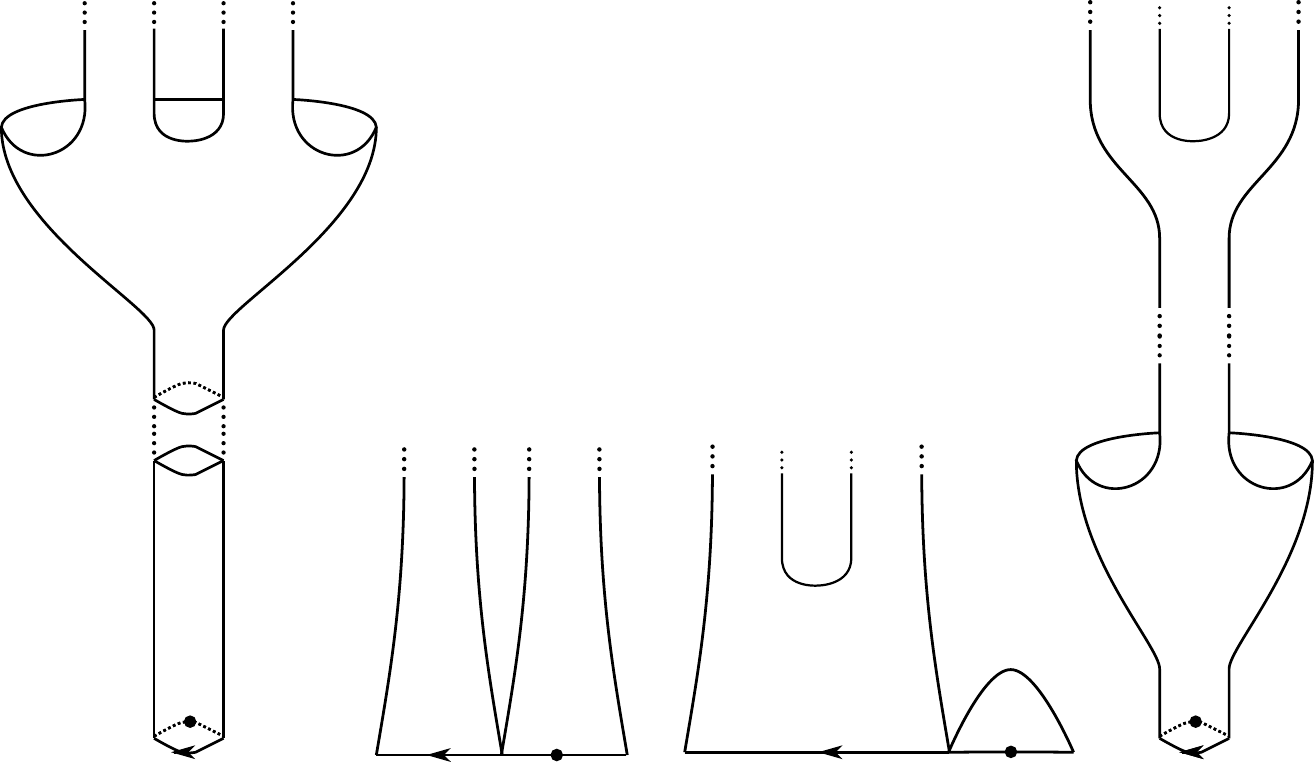}
  \caption{ }
  \label{fig:half_discs_2_inputs}
\end{figure}
\subsection{Floer data on annuli}
We start by making auxiliary choices to perturb the Cauchy-Riemann equation on an annulus:

\begin{defin} \label{def:floer_datum_annnuli}
A \emph{Floer datum} $D_{S}$ on a stable annulus $S \in \Annbar^{-}_{d}$ with $d$ positive boundary punctures $(\xi^1, \ldots, \xi^d) $  consists of the following choices on each component:
\begin{enumerate}
\item Time shifting map:  A map $ \rho_{S} \co \partial S \to [1,d]$ which is constant near each puncture and equals $1$ near the boundary component carrying the marked point.  We write $w_{k,S}$ for the value on the $k$\th end and set  
  \begin{equation}
    w_{0,S} = \sum_{k=1}^{d} w_{k,S}.
  \end{equation}
\item Basic $1$-form:  A closed $1$-form $\alpha_{S}$ whose restriction to the boundary vanishes and whose pullback under the $k$\th end agrees with $w_{k,S} dt$. 
\item Subclosed $1$-form: A $1$-form $\beta_{S}$ such that $d \beta_{S} \leq 0$, which vanishes near the ends  and agrees with a (constant) multiple of $\alpha$ near the outgoing segment
 \item Hamiltonian perturbations: a map $H_{S} \co S \to \sH(\TQ)$  such that the pullback of $ X_{S} \otimes \alpha_{S} $ under the $k$\th end agrees with  $  X_{\frac{H}{w_{k,S}} \circ \psi^{w_{k,S}}} \otimes dt$.  Moreover, near the outgoing segment on $S$ and the zero section in $\TQ$, we have
  \begin{equation} \label{eq:vanishing_inhomogeneous_dbar_boundary}
   X_{\frac{F \circ \psi^{w_{0,S}}}{w_{0,S}}} \otimes    \beta_{S} +  X_{S} \otimes \alpha_{S} = 0.
  \end{equation}
\item Almost complex structures:  A map $I_{S} \co S \to \sJ(\TQ)$ whose pullback under the $k$\th end agrees with  $ (\psi^{w_{k,S}})^{*} I_t $. 
\end{enumerate}
\end{defin}

If $S$ lies on the image of $\Pilbar_{d_1,1}   \times  \Pilbar_{d_2}$ as in Equation  \eqref{eq:boundary_moduli_annuli_coproduct_relation}, then we set $\beta_{S}  = 0$, and the restrictions of the universal Floer data $\bfD_{\cF} $  to $  \Pilbar_{d_2} $ and $ \Pilbar_{d_1}  $ determine the remaining data for $D_{S}$: the vanishing in Equation \eqref{eq:vanishing_inhomogeneous_dbar_boundary} is automatic because the restriction of $H_S$ to $\partial S$ agrees with $G$ which vanishes near $Q$ (this condition was imposed as part of Definition \ref{def:floer_datum_half_popsicle}).   If $d_2=d$, then $ \Pilbar_{0,1} $  may be identified with an infinite strip carrying a boundary marked point,  and we assume that $\alpha_{S}$ vanishes, while the almost complex structure is translation invariant and given by $I_{t}$.

On the other hand, if $S$ lies on the stratum (\ref{eq:boundary_moduli_annuli_interior_breaking}), we use $\bfD_{\OC}$ to define the Floer data on the component carrying the incoming boundary points, and use data conformally equivalent to the one fixed in the discussion preceding Equation (\ref{eq:boundary_disc_1_interior_1_boundary_output}) on the component carrying the marked point.  

\begin{defin}
A  \emph{universal and conformally consistent}  choice of Floer data for the homotopy  is a choice  $\bfD_{\cH}$ of Floer data  for every integer $d \geq 1$, and every representative of an element of  $\Annbar^{-}_{d}$ which vary smoothly over the compactified moduli space, such that the two natural Floer data on any irreducible component of a singular annulus are conformally equivalent, and which agree to infinite order with the data obtained by gluing near every boundary stratum.
\end{defin}

Given a sequence of chords $\vx = \{ x_1, \ldots, x_d \}$, we define the moduli space $ \Ann^{-}_{d}(\vx) $ to be the space of maps $u \co S \to M$ whose source is an arbitrary element of $  \Ann^{-}_{d} $, with asymptotic condition $  \psi^{w_{k,S}}  \circ  x_k  $ at the $k$\th  incoming end, which map the component carrying the marked point to $Q$ and the other boundary components to $\Tq$, and which solve the Cauchy-Riemann equation
\begin{equation}
\label{eq:dbar_equation_homotopy}
\left( du -   X_{H_S}   \otimes \alpha_{S} -  X_{\frac{F \circ \psi^{w_{0,S}}}{w_{0,S}}} \otimes \beta_{S} \right)^{0,1} = 0.
\end{equation}

\begin{lem}
For generic choices of Floer data $\bfD_{\cH}$, the Gromov bordification of $ \Ann^{-}_{d}(\vx)  $ is a compact manifold of dimension
\begin{equation*}
  d  - \sum_{k=1}^{d} |x_k|  
\end{equation*}
whose boundary decomposes into codimension $1$ strata which are the images of natural inclusions of the moduli spaces
\begin{align}
 \Discbar^{1}(y) \times \Discbar_{d}^{1}(y; \vx) & \quad y \in \Orbit \\ \label{eq:boundary_stratum_annuli_differential_CC_0}
 \Pilbar_{d_1,1} (q,\vx[1],q) \times   \Pilbar_{d_2} (q,\vx[2],q)   & \quad 0 \leq r < d_2 \leq d  = d_1 +d_2\\
\Annbar_{d_1}^{-}(\vx[1]) \times  \Discbar_{d_2}(x; \vx[2])  & \quad 1 \leq d_1 < d = d_1 + d_2 -1 \textrm{ and } x \in \Chord \label{eq:boundary_stratum_annuli_differential_CC}
\end{align}
where in the second type of stratum, $ \vx[1] = (x_{r+1}, \ldots, x_{r+d_1} ) $ and  $ \vx[2] = (x_{r+d_1+1}, \ldots, x_d, x_1, \ldots, x_{r}) $, while in the  last type of stratum $x$ agrees with one of the elements of $\vx[1] $, and the sequence obtained by removing $x$  from $\vx[1]$ and replacing it by the sequence $\vx[2]$ agrees with $\vx$ up to cyclic ordering.   \noproof
\end{lem}

Explicitly, we encounter two possibilities in Equation \eqref{eq:boundary_stratum_annuli_differential_CC}: either (1) there exists an integer $k$ such that  $\vx[1] = (x_1, \ldots, x_k, x, x_{k+d_2+1}, \ldots, x_d)$ and $\vx[2]=(x_{k+1}, \ldots, x_{k+d_2} )  $, or (2) there exists an integer $r$ such that  $\vx[1]=(x_{r+1},  \ldots, x_{r+d_1-1}, x )  $ and  $\vx[2]=(x_{r+d_1}, \ldots, x_{d}, x_1, \ldots, x_{r} )  $.    These are distinguished by whether $x_{d_2}$ lies in $\vx[1] $ or $ \vx[2] $.

\subsection{Orienting the moduli space of annuli}
As it is a manifold with boundary, the moduli space $ \Annbar_{d}^{-}(\vx)  $ admits a relative fundamental chain; we have already chosen such a relative fundamental chain for $  \Discbar^{1}(y) $ in Section \ref{sec:count-half-cylind}, for $ \Pilbar_{d}(q,\vx,q) $ in Section \ref{sec:a_infty-homomorphism}, and for  $ \Pilbar_{d,1}(q,\vx,q) $ in Section \ref{sec:adding-marked-point}.  In $\Lef_{*}(  \Annbar_{d}^{-}(\vx)  ) $,  these classes can be chosen for all sequences $\vx$ so that
\begin{multline} \label{eq:fundamental_class_annuli}
  \partial [  \Annbar_{d}^{-}(\vx)  ]  =  \sum (-1)^{(d-1)(n - |y|)}  [\Discbar^{1}(y)] \times [\Discbar_{d}^{1}(y; \vx)] +  \\
 \sum  (-1)^{ \frac{n(n+1)}{2}  +\diamond_{r}^{d_1}} [\Pilbar_{d_1,1} (q,\vx[1],q)]  \times  [\Pilbar_{d_2} (q,\vx[2],q)]   +  \sum_{x_d \in \vx[1]}  (-1)^{1+\sharp}   [\Annbar_{d_1}^{-}(\vx[1])] \times  [\Discbar_{d_2}(x; \vx[2])] \\
+ \sum_{x_d \in \vx[2]}  (-1)^{d_1 +1  +   \diamond_{r}^{d_1} + d_{2} |\vx[2]| }   [\Annbar_{d_1}^{-}(\vx[1])] \times  [\Discbar_{d_2}(x; \vx[2])]
\end{multline}
where the signs $\sharp$ is given by Equation  \eqref{eq:sign_sharp}, and the new sign is
\begin{equation} \label{eq:diamond_sign}
  \diamond_{r}^{d_1}=  r(d+1) + \left(\sum_{k=1}^{r} |x_{k}| \right)  \cdot \left(\sum_{k=r+1}^{d} |x_{k}| \right)  +d_{2}\left(\sum_{k=r+1}^{r+d_1} |x_{k}| \right).
\end{equation}
The proof of the existence of such classes proceeds by induction: in the inductive step, one has to prove that the right hand side of Equation \eqref{eq:fundamental_class_annuli} is closed.  The analogue for the moduli space of half-discs is Equation \eqref{eq:boundary_fundamental_chain}, which we can generalise to our setting using the choice of fundamental chains  on half-discs with a marked point on the outgoing segment which we fixed in Equation \eqref{eq:product_moduli_marked_point}.

To prove the correctness of the signs, one must compare various product orientations with those induced at the boundary of the moduli space.  The computation for each term is separate, but since there is no fundamental difference between them, we shall only illustrate one of them.  The starting point is the fact that the orientation on $ \Annbar_{d}^{-}(\vx)   $ comes from a canonical up to homotopy isomorphism (see, e.g. Lemma C.4 of \cite{generation} and Appendix \ref{sec:signs-orientations} for the generalisation to the relatively $\Spin$ case)
\begin{equation} \label{eq:orientation_annuli}
  \lambda( \Annbar_{d}^{-}(\vx)    ) \cong \lambda(  \Annbar_{d}^{-} ) \otimes \lambda(\Tq) \otimes \lambda(Q) \otimes o_{x_1}^{-1} \otimes \cdots \otimes o_{x_d}^{-1}.
\end{equation}
If we now consider the second term on the right hand side of Equation \eqref{eq:fundamental_class_annuli}, the orientations of the factors come from isomorphisms
\begin{align*}
 \lambda( \Pilbar_{d_1,1} (q,\vx[1],q)) & \cong  \lambda( \Pilbar_{d_1} ) \otimes o_{q} \otimes o_{x_{r+1}}^{-1} \otimes  \cdots \otimes o_{x_{r+d_1}}^{-1} \otimes o_{q}^{-1} \otimes \lambda([0,1])\\
\lambda(\Pilbar_{d_2} (q,\vx[2],q)) & \cong  \lambda( \Pilbar_{d_2} ) \otimes o_{q}^{-1}  \otimes o_{x_{r+d_1+1}}^{-1} \otimes \cdots \otimes o_{x_{d}}^{-1} \otimes o_{x_{1}}^{-1} \otimes  \cdots \otimes o_{x_{r}}^{-1} \otimes o_{q}^{-1}.
\end{align*}
Taking the tensor product of these two expressions gives the product orientation.  In order to arrive at Equation \eqref{eq:orientation_annuli} we must (i)  move each copy of $o_{q}$ next to $o_{q}^{-1}$, and cancel them  (ii) move the copy of $\lambda([0,1]) $ next to $\lambda( \Pilbar_{d_1} )   $  (iii) move $  \lambda( \Pilbar_{d_2} ) $ to be adjacent to $  \lambda( \Pilbar_{d_1} ) $ (iv) identify the sign difference between the product orientation on the abstract moduli spaces $ \Pilbar_{d_1,1}  \times  \Pilbar_{d_2}$ and its boundary orientation (v) reorder the (inverse) orientation lines associated to the inputs and (vi) identify $\lambda(\Tq) \otimes \lambda(Q)  $ with the trivial line.  Since the gradings have been chosen in such a way that $q$ has degree $0$, the first operation does not contribute any Koszul sign, while the parity associated to the others is given by the following sum whose terms correspond in order to the operations listed above:
\begin{equation*} 
\sum_{k=r+1}^{r+d_1} |x_k|  + (d_{2}-1)\left(  \sum_{k=r+1}^{r+d_1} |x_k|  \right)  + (d-1) r + \left(  \sum_{k=1}^{r} |x_k|  \right) \left(  \sum_{k=r+1}^{d} |x_k|  \right) + \frac{n(n+1)}{2}.
\end{equation*}
The appearance of the last term is explained in the proof of Lemma 6.8 in \cite{generation}.   Note that this sum exactly reproduces the sign that appears in the second line of the right hand side in  Equation \eqref{eq:fundamental_class_annuli}.

Having chosen fundamental chains, we define a map
\begin{equation*}
  \cH \co CC_{*}( CW^{*}_{b}(\Tq)) \to C_{-*}( \sL Q )
\end{equation*}
by linearly extending the formula
\begin{equation}  \label{eq:sign_homotopy}
  x_{d} \otimes \ldots \otimes x_{1} \mapsto (-1)^{|x_d|+\dagger + d |\vx|} \ev_{*}( [  \Annbar_{d}^{-}(\vx)  ]  ) .
\end{equation}
In characteristic $2$, the fact that $\cH$ is a homotopy between the two compositions in Diagram \eqref{eq:commutative_diagram_symplectic_string} follows from Equation \eqref{eq:fundamental_class_annuli}.  The left hand side is $\partial \circ \cH$, while the first term on the right corresponds to $\G \circ  CC_{*}(\cF)  $, the second to $  \CL  \circ  \OC $, and the last two to $\cH \circ b$ ($b$ is the Hochschild differential from Equation \eqref{eq:hochschild_differential}). Since $\G$ was defined separately on words of length $1$ and $2$, we note that the  composition of $\G$ with the component of $CC_{*}(\cF ) $  whose image consists of words of length $1$ corresponds to the case where $d_1=0$ in the first term of the right hand side.  Indeed, the only element of  $\Pilbar_{0,1} (q,\emptyset,q)  $ is a constant triangle at $q$, and $\G \circ \cF(\vx)$ is then equal to the image of the fundamental chain $\Pilbar_{d} (q,\vx,q)  $ in the chains over the free loop space.  The case $d_1 \neq 0$ recovers the terms in $\G$ defined on words of length $2$.

\begin{lem}
$\cH$ is a homotopy between $(-1)^{\frac{n(n+1)}{2}} \G \circ  CC_{*}(\cF)  $ and $ \CL  \circ  \OC$.
\end{lem}
\begin{proof}[Sketch of proof:]
Again, we shall only verify that the sign for the contribution of $ \G \circ  CC_{*}(\cF)  $ is correct.  In fact, since $  \G  $  was defined separately for words of length $1$ and $2$, we shall only do this in a further special case when $d_1 \neq 0 $.  So our goal is to go through the signs in the definition of $  \G^{(2)} \circ  CC_{*}(\cF)  $, and ensure that, together with  Equation \eqref{eq:diamond_sign}, they add up precisely to the sign appearing in Equation \eqref{eq:sign_homotopy}.  Start by noting that
\begin{equation*}
  \G( [\Pilbar_{d_2} (q,\vx[2],q)] ,  [\Pilbar_{d_1} (q,\vx[1],q)]   ) = (-1)^{ d_2 - 1 + | \vx[2]|  } \ev_{*}(  [\Pilbar_{d_1,1} (q,\vx[1],q)]  \times  [\Pilbar_{d_2} (q,\vx[2],q)] )
\end{equation*}
coming from the fact that the factor $[0,1]$ appears last in Equation \eqref{eq:moving_basepoint} instead of just before $\Pilbar_{d_2} (q,\vx[2],q)$.  Next, see from Equation \eqref{eq:CC_map} that 
\begin{equation*}
CC_{*}^{(2)}(\cF)(a_d, \ldots, a_1) = \sum (-1)^{ \shuffle_{r}^{r+d_1} } \cF^{d_2}(a_{r},  \ldots, a_{1} , a_{d}, \ldots, a_{r+d_1+1})   \otimes \cF^{d_1}(a_{r+d_1 }, \ldots, a_{r+1} ) .
\end{equation*}
Finally, the definitions of $\cF^{d_1}$ and $\cF^{d_2}$ incorporate, via Equation \eqref{eq:twist_A_infty_homomorphism} the signs
\begin{align*}
  \dagger_{1} + d_1 | \vx[1]| &= \sum_{k=r+1}^{r+d_1} (r+k + d_1) |x_k|  \\
\dagger_{2} + d_2 | \vx[2]| & = \sum_{k=1}^{r} (r+k) |x_k| +   \sum_{k=r+d_1+1}^{d} (d+r+k) |x_k|.
\end{align*}
The reader can now directly check that
\begin{equation*}
    |x_d|+\dagger + d |\vx| = \diamond_{r}^{d_1} + d_2 - 1 + | \vx[2]|  + \shuffle_{r}^{r+d_1} +  \dagger_{1} + d_1 | \vx[1]|  + \dagger_{2} + d_2 | \vx[2]|.
\end{equation*}
\end{proof}
\appendix

\section{Orientations induced by relative $\Spin$ structures} \label{sec:signs-orientations}
In this Section, we discuss orientations of moduli spaces of holomorphic discs with interior punctures and of holomorphic annuli; these appear in the construction of the maps $\OC$ and $\CO$, as well the map $\cH$.  In the absence of the background class $b$ these orientations were constructed in \cite{generation}, while the case of holomorphic discs was discussed in \cite{string-top}.

We consider a general situation, in which $M$ is a symplectic manifold with vanishing first Chern class, $\Orbit$ a collection of (Hamiltonian) orbits in $M$, $L$ and $Q$ a pair of graded Lagrangian submanifolds which are relatively $\Spin$ for the same background class $b \in H^{2}(M, \bZ_{2})$ which is the second Stiefel-Whitney class of a vector bundle $E_{b}$,  and $\Chord$ the set of (Hamiltonian) chords with endpoints on $L$.  We choose relative $\Spin$ structures on $L$, $Q$, each chord $x \in \Chord$ (i.e. on the associated Lagrangian path $\Lambda_{x}$) and on intersections of $L$ and $Q$ thought of as chords for the trivial Hamiltonian.  

We remind the reader of the following result which we will repeatedly use:
\begin{lem} \label{lem:equivalence_Spin_structures}
For each vector bundle $E$,  there is a bijection between the set of  $\Spin$ structures on $E$ up to isomorphism and those on $E' \oplus E$, obtained by taking the direct sum with a fixed $\Spin$ structure on a vector bundle $E'$.  \noproof
\end{lem}

Let us now review the construction of the isomorphisms required to define the map in Equation \eqref{eq:iso_orientation_differential} and its analogue for discs with one output and an arbitrary number of inputs.  Assume that $S$ has $d$ positive ends, whose images under $u$ converge to chords $x_{1}, \cdots, x_{d}$, and one negative end converging to $x_0$.  For each such asymptotic end $x_k$, the contractibility of the disc equips $(x_k \circ t)^{*}(E_{b})$ with a unique $\Spin$ structure up to isomorphism, which we can then restrict to the boundary.  In particular, the data of Equation \eqref{eq:Pin_structure_chord} is equivalent to choosing a $\Spin$ structure on $\Lambda_{x_k}$, which is the boundary condition of the operator $D_{x_k}$ associated to the chord $x_k$.

Now, the main point is that as long as $S$ is contractible, the pullback of $E_{b}$ under a map $u \co S \to M$ still admits a unique $\Spin$ structure up to isomorphism, which induces such a $\Spin$ structure on the restriction to each boundary component.  Assuming the boundary maps to $L$ under $u$, Lemma \ref{lem:equivalence_Spin_structures} together with the data fixed in Equation \eqref{eq:rel_pin_structure} determine a $\Spin$ structure on the pullback by $u|\partial S$ of the tangent space of $L$.

Having produced $\Spin$ structures on the boundary conditions of the linearisation $D_{u}$ of the Cauchy-Riemann operator at $u$, we can use by now standard methods (see, e.g. Section 11 of \cite{seidel-book}) to produce the isomorphism of Equation \eqref{eq:isomorphism_disc}:   one glues the linearisation $D_{u}$ of the Cauchy-Riemann operator at $u$ to the operators $D_{x_k}$ associated to all inputs to obtain an operator
\begin{equation} \label{eq:glued_operator}
  D_{u} \# D_{x_d} \# \cdots \# D_{x_1} 
\end{equation}
on a disc with one (outgoing) end.   The boundary conditions of this operator carry a $\Spin$ structure from the above considerations, and the fact that the Lagrangian $L$ is graded implies that if we trivialise the associated vector bundle over the disc on whose sections it acts, then this operator is homotopic, through the space of Fredholm operators, to $D_{x_0}$.  Up to homotopy, the choice to be made is a  deformation of the boundary conditions to those of the operator $D_{x_0}$; there are two such choices, corresponding to the fact that the fundamental group of the space of based paths on the Grassmannian of Lagrangians is isomorphic to $\bZ_{2}$.  Each of these choices induces an isomorphism in Equation \eqref{eq:isomorphism_disc} and we fix the isotopy to be the one whose associated family of boundary conditions carries a $\Spin$ structure which restricts, at both ends, to those just fixed for $\Lambda_{x_0} $  and for the boundary conditions of the glued operator in Equation \eqref{eq:glued_operator}.

Exactly the same construction yields orientations for the moduli spaces of half-discs with boundary on $L$ and $Q$, which were used in Section \ref{sec:moduli-space-half} to construct a map from the wrapped Floer complex to the algebra of chains on the based loop space.

However, the part of this argument relying on the uniqueness of the $\Spin$ structure on the pullback of $E_{b}$ fails for a general Riemann surface which is not contractible.  Nonetheless, a $\Spin$ structure on the restriction of this pullback to a subset whose inclusion induces an isomorphism on cohomology is equivalent to such a structure on  the entire surface.  We already used this idea in Section \ref{sec:constr-twist-sympl} when constructing the differential computing a twisted version of symplectic cohomology.  We shall implement it more generally for half-cylinders and annuli:  Let $y$ be an orbit and let $\vx=\{ x_1, \ldots, x_d \}$ be chords  with endpoints on $L$:
\begin{lem}
A choice of relative $\Spin$ structure on $Q$ induces an orientation of the moduli spaces $\Discbar^{1}(y)$ relative to $  o_{y} \otimes  \kappa_{y}^{b}$  (see Equation \eqref{eq:iso_tangent_half_cylinder}).  Choices of (i) a $\Spin$ structure on $y^{*}(E_{b})$, (ii) a relative $\Spin$ structures on $L$ and on the chords $x_k$ as in \eqref{eq:Pin_structure_chord}, and (iii) orientations on the abstract moduli spaces induce orientations of the moduli spaces  $ \Discbar_{d}^{1}(y; \vx)$ and  $ \Ann_{d_1}^{-}(\vx[1]) $ relative to $o_{x_d} \otimes \cdots o_{x_1}$.
\end{lem}
\begin{proof}
We shall omit the easier case of  $\Discbar^{1}(y)$, and focus on the key point in the other two cases: producing from the above data an orientation of the determinant line of the operator obtained by gluing the operators associated to the inputs and outputs to the linearised Cauchy-Riemann operator at a given element of the moduli space.

We first consider an element of $   \Discbar_{d}^{1}(y; \vx)$, i.e. a map
\begin{equation*}
  u \co S \to M
\end{equation*}
from a disc with an interior (negative) puncture converging to $y$ and $d$ boundary (positive) punctures converging to the chords listed in the sequence $\vx$.  We may glue the operator $D_{x_k}$ at the $k$\th incoming puncture to the linearisation $D_{u}$ of the $\dbar$ operator at $u$, and obtain a glued operator
\begin{equation} \label{eq:glued_operator_half_disc}
  D_{u} \# D_{x_d} \# \cdots \# D_{x_1}
\end{equation}
on a disc $T$ with one interior marked point.  Letting $\lambda$ stand for the top exterior power of a vector space, the determinant line of this operator is naturally isomorphic to
\begin{equation} \label{eq:gluing_det_lines}
  \det(D_{u}) \otimes o_{x_d} \otimes \cdots \otimes o_{x_1}  \cong \lambda( \Discbar_{d}  )^{-1} \otimes  \lambda(\Discbar_{d}^{1}(y; \vx)) \otimes o_{x_d} \otimes \cdots \otimes o_{x_1}.
\end{equation}

The pullback of $E_{b}$ under $u$ defines a vector bundle over $T$.  A $\Spin$ structure on this vector bundle is induced by a $\Spin$ structure on $y^{*}(E_{b})$ since the inclusion of a neighbourhood of the boundary puncture is homotopy equivalent to $T$.  Using Lemma \ref{lem:equivalence_Spin_structures}, we obtain a $\Spin$ structure on the boundary conditions of  the glued operator \eqref{eq:glued_operator_half_disc}.  We can now use standard results (e.g. Proposition 11.13 of \cite{seidel-book}) to produce a natural isomorphism
\begin{equation}
  \det( D_{u} \# D_{x_d} \# \cdots \# D_{x_1}  ) \cong o_{y} \otimes \lambda( TL).
\end{equation}
Note that $\kappa^{b}_{y}$ does not appear in this equation even though we asserted that the moduli space was oriented relative  $  o_{y} \otimes  \kappa_{y}^{b}$:  this is because a trivialisation of $ \kappa^{b}_{y}$, i.e. a $\Spin$ structure on $y^{*}(E_b)$, was already used to produce this isomorphism, and changing the $\Spin$ structure will change the isomorphism by a sign.

Combining this isomorphism with that of Equation \eqref{eq:gluing_det_lines} gives the desired (relative) orientation of $\Discbar_{d}^{1}(y; \vx)$.

We shall now reduce the case of annuli to that of punctured discs: so we consider a map 
\begin{equation*}
  u \co S \to M
\end{equation*}
representing an element of $\Ann_{d_1}^{-}(\vx[1])   $.  In particular, we may glue the operator $D_{x_k}$ at the $k$\th incoming puncture to the linearisation $D_{u}$ of the $\dbar$ operator at $u$, and obtain a glued operator $D_{\bar{S}}$ on an annulus $\bar{S}$ with one boundary component free of marked points and punctures, and the other carrying one marked point.

At this stage, we let the modular parameter of $\bar{S}$ go to infinity, and obtain a degeneration of $\bar{S}$ into two discs $T_1$ and $T_2$ each carrying exactly one puncture which lies in the interior (the surface appearing on the right in Figure \ref{fig:half_discs_1_input} shows the corresponding degeneration if we add a puncture to one of the two discs).  Choosing a deformation of the Cauchy-Riemann operator along this path of conformal structures on the annulus (and keeping the boundary conditions unchanged) we obtain an isomorphism
\begin{equation} \label{eq:split_annulus_two_discs_interior_det}
  \det(D_{\bar{S}}) \cong  \det(D_{T_1})  \otimes \det(D_{T_2}).
\end{equation}
The pullback of $E_b$ under $u$ deforms to vector bundles $E_1$ and $E_2$ on the manifolds $\bar{T}_1$ and $\bar{T}_2$ obtained by adding a circle at infinity along the interior end of $T_1$ and $T_2$, and the restrictions of the two vectors bundles to the circle at infinity are naturally isomorphic.  Using Lemma \ref{lem:equivalence_Spin_structures}, a choice of $\Spin$ structure on the restriction of $E_1$ (and hence $E_2$) to this circle induces a $\Spin$ structure on the boundary conditions of the operators $  D_{T_1}$  and $D_{T_2}$, and hence, as in the above discussion in the case $ u \in   \Discbar_{d}^{1}(y; \vx)$, an orientation of $ \det(D_{T_1}) $ and $\det(D_{T_2}) $, relative to the tangent spaces of $Q$ and $L$.

While the individual orientations on $ \det(D_{T_1}) $ and $\det(D_{T_2}) $ depend on this additional choice, changing the $\Spin$ structure on the restriction of $E_1$ to this circle reverses both orientations.  In particular, an orientation on $   \det(D_{\bar{S}})  $  is canonically determined, via Equation \eqref{eq:split_annulus_two_discs_interior_det}, by the data listed in the statement of the Lemma.
\end{proof}
\begin{rem}
 Having used the data introduced in Sections \ref{sec:geom-prel} and \ref{sec:constr-twist-sympl} to orient all the moduli spaces which appear in this paper, one might expect that we continue with a sign analysis verifying that the sign conventions in \cite{generation} are still valid in the twisted setting we consider.  This is rendered unnecessary by the following fact:  say $u_1$ and $u_2$ are elements of some moduli spaces $\Mod_1$ and $\Mod_{2}$  of holomorphic curves which can be glued to form a  broken curve $u=u_{1} \# u_2$ in the boundary of another moduli space $\Mod$.  The signs appearing in Floer theory come from (1) Koszul signs which are introduced when rearranging the tensor product of the isomorphisms which give relative orientations of the tangent spaces $T_{u_1} \Mod_1$ and $ T_{u_2} \Mod_2$ (e.g. Equation \eqref{eq:orientation_bundles_half-discs}) to yield the isomorphism giving the relative orientation of $ T_{u} \Mod$  and (2) any remaining difference in orientation between the natural orientation of a product and that of a boundary.  

 In the twisted case, all our constructions of orientations use the pullbacks $u_{1}^{*}(E_{b})$, $u_{2}^{*}(E_{b})$, and $u^{*}(E_{b})$ to reduce to the case of $\Spin$ boundary conditions.  Since there is a natural isomorphism between $u^{*}(E_{b})$ and the result of gluing  $u_{1}^{*}(E_{b})$ and $u_{2}^{*}(E_{b})$, no new sign arises because of differences between product and boundary orientations.

 At the same time, the choice of $\Spin$ structure on a pullback of $E_{b}$ appear as a new factor in the formula for the relative orientation of a moduli space (see e.g. $\kappa_{y}^{b}$ in Equation \eqref{eq:iso_tangent_half_cylinder}).  These are naturally graded in degree $0$ since they do not change the degree of the corresponding generator of symplectic cohomology.  In particular, they can be permuted freely in expressions like Equation \eqref{eq:orientation_bundles_half-discs} without the appearance of any new sign.
\end{rem}

Our final remark in this Section concerns the sign conventions in \cite{AS} which need to be corrected in order for Proposition \ref{prop:maps_iso} to be valid:
\begin{rem}
Abbondandolo and Schwarz define symplectic cohomology with $\bZ$ coefficients by implementing the coherent orientations of \cite{HF} in the setting of cotangent bundles.  They choose trivialisations  $x^{*}(T \TQ)$ for each Hamiltonian chord $x$, which they require to be induced by a trivialisation of the vertical distribution by complexifying, i.e. the vertical distribution is mapped to $i \bR^{n} \subset \bC^{n}$.  Their trivialisations lie in the same homotopy class as those we alluded to in Section \ref{sec:constr-twist-sympl}, and which are induced by a complex form on $T \TQ$ obtained by complexifying a real volume form on $Q$:  this is true essentially because the vertical distribution has constant phase with respect to such a volume form.  In order to obtain a chain map relating Floer and Morse theory, the solution is quite simple: either twist the contribution of each cylinder to the differential in Hamiltonian Floer cohomology by a sign which vanishes if and only if the trivialisation of the vertical sub-bundle fixed at both end extends to the cylinder, or twist the Morse homology of the loop space by a local system.    The first solution recovers a Hamiltonian Floer homology group canonically isomorphic to our twisted symplectic cohomology group $SH^{*}_{b}(\TQ)$.

In order to understand why the untwisted version of the construction in \cite{AS} agrees with our untwisted symplectic cohomology group, we briefly discuss orientations. Given a solution $u$ to Floer's equation (\ref{eq:dbar_cylinder}) with asymptotic conditions at orbits $x$ and $y$, there is, up to homotopy,  a unique trivialisation of $u^{*}(T \TQ)$ which agrees with  the trivialisations fixed at the end.  Having linearised the problem, the theory of coherent orientations developed by Floer and Hofer in \cite{HF} is then used in \cite{AS} to associate signs to each cylinder. To conclude the desired isomorphism, use the essential uniqueness of coherent orientations  (e.g. Theorem 12 of \cite{HF}) to work with the following conventions which are close to that used in Section  \ref{sec:constr-twist-sympl}:  choose for each chord $y$ an orientation on the determinant line of some Cauchy-Riemann operator $D_{y}$ on the plane which agrees with the linearisation of  Equation \eqref{eq:dbar_cylinder} with asymptotic condition $y$ at infinity in the chosen trivialisation (i.e. trivialise the free abelian group $|o_{y}|$ appearing in Equation \eqref{eq:map_orientation_lines}).  Gluing this operator to linear Cauchy-Riemann operators on cylinders, we obtain coherent orientations as in \cite{HF}, proving that the Hamiltonian Floer homology group in \cite{AS} is isomorphic to our symplectic cohomology group for trivial background class.

\end{rem}

\section{Hitting the fundamental class} \label{sec:hitting}
In the setting of simplicial sets, Goodwillie constructed an isomorphism  between the Hochschild homology of the chains of the based loop space and the homology of the free loop space.  As noted in the introduction, this leads us to the expectation that the map defined in Section \ref{sec:an-ad-hoc} is a quasi-isomorphism.   However, we prefer to avoid delving into a comparison theorem between simplicial sets and the non-standard cubical models for homology used in this paper.  Instead, by factoring the inclusion of  $H_{*}(Q)$ in  $H_{*}(\sL Q)$ through Hochschild homology, we shall prove the following result, which is all that is required:
\begin{lem}
\label{lem:fundamental_class_in_image}
The fundamental class of $Q$ lies in the image of $ H^{*}( \G) $.
\end{lem}

Giving an explicit model for the maps introduced by Adams in \cite{adams}, we define
\begin{align*}
 a^{k-1} \co [0,1]^{k-1} & \to \Omega_{[0], [k]} \Delta^{k}\\
a^{k-1} (t_1, \ldots, t_{k-2}, t_{k-1})(s) & = \begin{cases} (1,s, 0, \ldots, 0, 0) & 0 \leq s \leq t_1 \\
(1, t_1, s - t_1 , \ldots, 0, 0) &  t_1 \leq s \leq t_1 + t_2 \\
\cdots & \cdots \\
(1, t_1, t_2 , \ldots, t_{k-1}, s - \sum_{j=1}^{k-1} t_j)  &  \sum_{j=1}^{k-1} t_j \leq s \leq 1 +  \sum_{j=1}^{k-1} t_j.
  \end{cases}
\end{align*}
The reader unfamiliar with the intuition behind these formulae should consult Figure \ref{fig:adams_paths}.
\begin{figure}
\centering
\includegraphics{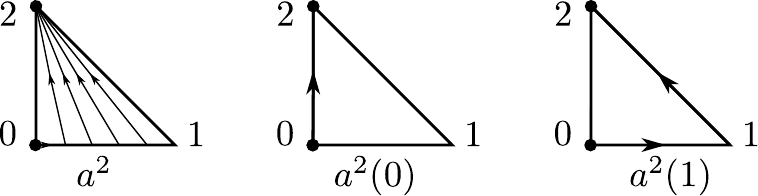}
\caption{}
\label{fig:adams_paths}
\end{figure}

Recall that a simplicial triangulation of $Q$ is a triangulation in which the vertices are totally ordered, and such that every cell may be uniquely represented by a sequence of vertices which is increasing with respect to this order.  Let us pick a subdivision of $Q$ into simplices by collapsing a maximal tree from a simplicial triangulation, and write $q$ for the unique resulting vertex.  In particular, every $k$-cell of this subdivision is still uniquely determined by an increasing sequence of the vertices of the original triangulation.  Every such  cell  $\U = [u_{0}, \ldots, u_{k}] $ determines a map $ \sigma_{\U} \co \Delta^{k} \to Q$, and hence a cubical chain in the based loop space:
\begin{align}
  \tau_{\U} \co [0,1]^{k-1} & \to \Omega_{q} Q \\
\tau_{\U} & \equiv \sigma_{\U} \circ a_{k-1}.
\end{align}
More generally, we shall consider sequences which become increasing after a cyclic reordering:  given such a reordering $\U['] = [u_i, \ldots, u_k, u_0, \ldots, u_{i-1}]  $ of $\U$, we obtain a different map $\tau_{\U[']} $ by composing $ \sigma_{\U} \circ a_{k-1} $  with the automorphism of the simplex that cyclically reorders the vertices.

Adams essentially observed that these chains satisfy the following inductive relation:
\begin{equation}
  \label{eq:boundary_adams_chains}
  \partial  \tau_{\U} = \sum_{0 < j < k} (-1)^{j} \tau_{\U - \{u_j\}} + \sum_{\substack{\U[1] = [u_0, \ldots, u_r] \\ \U[2] = [u_r, \ldots, u_k]} } \mu_{2}^{\P}(\tau_{\U[1]}, \tau_{ \U[2]} ).
\end{equation}
Given an integer $r$ between $0$ and $k$, we obtain a family of loops  in $\Delta^{k}$ based at $[0]$ by concatenating  (1) the composition of $a^{r-1}$ with the inclusion of the face $[0, \ldots, r]$ and (2) the composition of $ a^{k-r} $ with the face $[r, \ldots, k, 0]$.  Omitting the inclusion of faces from the notation, we consider a contraction of this family of loops to the basepoint
\begin{align}
c^{k}_{r} \co [0,1]^{k} & \to \Omega_{[0]} \Delta^{k}\\
c^{k}_{r} (t_1, \ldots, t_{k}) & = t_{k} (1,0, \ldots, 0  ) + (1-t_k) \left( a^{r-1}(t_1, \ldots, t_{r-1}) \cdot  a^{k-r}(t_r, \ldots, t_{k-1}) \right)   \label{eq:formula_contracting_homotopy}
\end{align}
\begin{figure}
 \centering
  \includegraphics{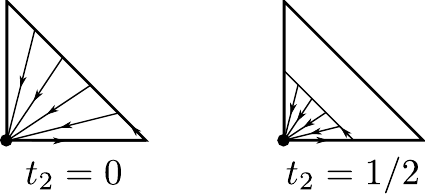}
  \caption{ }
  \label{fig:contract_concatenate}
\end{figure}
Figure \ref{fig:contract_concatenate} shows the restriction of $c^{1}_2$ to $t_2=0,1/2$.

This construction defines families of loops parametrised by appropriate pairs of cells in a simplicial triangulation.  Explicitly, given two cells $\U[1] = [u_0, \ldots, u_r]$ and $\U[2] = [u_r, \ldots, u_k, u_0]  $ with $r$ is an integer between $0$ and $k$ such that $ [u_0, \ldots, u_k]$ is a cell in $Q$, we write
\begin{equation}
  \label{eq:cap_cells}
  \U[1] \wedge \U[2]  =  [u_0, \ldots, u_k],
\end{equation}
and define a cubical chain in the loop space
\begin{align}
  \pi_{(\U[1], \U[2])} \co  [0,1]^{k} & \to \Omega_{q} Q \\
\pi_{(\U[1], \U[2])} & \equiv \sigma_{ \U[1] \wedge \U[2]} \circ  c^{k}_{r}.
\end{align}
Given a triple $ \V[1] $,  $\V[2]$, and $ \V[3] $  whose initial and final vertices agree cyclically, we may similarly construct a family of loops
\begin{equation*}
  \pi_{(\V[1], \V[2], \V[3])} \co  [0,1]^{k-1} \to \Omega_{q} Q 
\end{equation*}
by concatenating the paths associated to the three cells, and using the last coordinate to contract to the starting point of $ \V[1]$.  Again, we note that these maps make sense even if we cyclically reorder the vertices.

We desire a formula for the boundary of $  \pi_{(\U[1], \U[2])}  $  which is analogous to Equation \eqref{eq:boundary_adams_chains} for the chains $  \tau_{\U} $.  If $k=1$, the chain $ \pi_{[u_0, u_1], [u_1, u_0]}  $ is one dimensional, and corresponds to the family of paths  which start at $u_0$  move along the edge $[u_0,u_1]$ then turn back towards $u_0$.   The boundary consists of the constant path at $u_0$ and the concatenation of the paths from $u_0 $ to $u_1$ and back. We conclude that
\begin{equation}
  \partial \pi_{[u_0, u_1], [u_1, u_0]} = [u_0] - \tau_{[u_0,u_1]} \cdot  \tau_{[u_1,u_0]} .
\end{equation}
If $k \geq 2$, the boundaries of the cubical chains $  \pi_{(\U[1], \U[2])}  $ are given by:
\begin{multline}
  \label{eq:contract_adams_chains}
  \partial  \pi_{\U[1], \U[2]} =  (-1)^{k-r+1} \mu_{2}(\tau_{\U[2]},   \tau_{\U[1]} ) + \sum_{0 < j < r}  (-1)^{j} \pi_{\U[1] - \{u_j\},  \U[2]} + \sum_{r < j \leq k} (-1)^{j+1}  \pi_{\U[1] ,  \U[2] - \{u_j\}} \\
+ \sum_{\substack{ \V[1] = [u_0, \ldots, u_j]  \\ \V[2] =  [u_j, \ldots, u_r] }}  (-1)^{j+1}   \pi_{(\V[1], \V[2], \U[2])} + \sum_{\substack{ \V[2] = [u_r, \ldots, u_j] \\ \V[3] =  [u_j, \ldots, u_0] } }   (-1)^{r+j+1}  \pi_{(\U[1], \V[2], \V[3])} 
\end{multline}  
The first term comes from the boundary facet $t_k=0$ in Equation \eqref{eq:formula_contracting_homotopy}, and the remaining term are essentially a consequence of Equation \eqref{eq:boundary_adams_chains} which describes the boundaries of the chains that Adams constructed.

Given a cell $\U = [u_0, \ldots, u_k]$, we consider the element $T(\sigma_{\U})$ of the cyclic bar complex of $ C_{-*}( \Omega_{q}Q) $ given by the sum
\begin{multline}
\label{eq:sum_adams}
 \sum_{0 \leq  j_1 < j_2 < \ldots <  j_d \leq k} (-1)^{d+ j_1 (k+1)  + j_d} \tau_{\U_d} \otimes \cdots \otimes \tau_{\U_1} + \\
\sum_{0 \leq  j_1 < j_2  \leq k}   (-1)^{j_1 k  +j_1 +  j_2 + 1}  \pi_{(\U[1], \U[2])} + (-1)^{(j_2 +1)( k+1)   - j_1 }   \pi_{(\U[2], \U[1])} \end{multline}
where  $\U_i = [u_{j_i}, \ldots, u_{j_{i+1}}]  $ if $i \neq d$ and $\U_d = [ u_{j_d}, \ldots, u_{k}, u_{0}, \ldots, u_{j_1}] $.  In particular, all the cells respect the original order, except possibly for $U_{d}$ in the first line, and $U_{2}$ in the second.
\begin{lem}
Equation \eqref{eq:sum_adams} defines a chain map  \begin{align}
    T \co C_{-*}(Q) & \to CC_{*}(C_{-*} \Omega_{q}Q  ) \\
 \sigma_{\U} & \mapsto T(\sigma_{\U}).
  \end{align}
\end{lem}
\begin{proof}[Sketch of proof] 
We shall explain the proof ignoring signs.  First, note that the boundary of $\tau_{\U_d} \otimes \cdots \otimes \tau_{\U_1}   $  in the cyclic bar complex is given by
\begin{multline} \label{eq:differential_sequence_simplices}
 \sum_{i=1}^{d}   \tau_{\U_d} \otimes \cdots \otimes \mu_1(\tau_{\U_i}) \otimes \cdots \otimes \tau_{\U_1}  + \sum_{i=1}^{d-1}  \tau_{\U_d} \otimes \cdots \otimes \mu_{2}(\tau_{\U_i}, \tau_{\U_{i+1}}) \otimes \cdots \otimes \tau_{\U_1}  \\
 + \mu_{2}(\tau_{\U_1},  \tau_{\U_d})  \otimes \tau_{\U_{d-1}}  \otimes \cdots \otimes \tau_{\U_2}.
\end{multline}
We claim the sum of that these expressions over all possible sequences $\U_1, \ldots, \U_d$ where the last element of $\U_i$ agrees with the first element of $\U_{i+1}$  is equal to the sum of those terms in $T \left( \partial \sigma_{U} \right)$ consisting of words of length greater than $1$.  To see this, we first note that, if $d>2$, we can form a new sequence by applying the operation of Equation \eqref{eq:cap_cells} to two successive elements.   Using Equation \eqref{eq:boundary_adams_chains}, we find that 
\begin{equation*}
  \tau_{\U_d} \otimes \cdots \otimes \mu_1 \left( \tau_{\U_i \wedge \U_{i+1}} \right) \otimes \cdots\otimes \tau_{\U_1} 
\end{equation*}
contributes exactly one term which cancels with
\begin{equation*}
 \tau_{\U_d} \otimes \cdots \otimes \mu_{2}(\tau_{\U_i}, \tau_{\U_{i+1}}) \otimes \cdots \otimes \tau_{\U_1}.
\end{equation*}
Similarly, the last term in Equation \eqref{eq:differential_sequence_simplices} cancels with one of the terms in 
\begin{equation*}
 \mu_1 \left( \tau_{\U_d\wedge \U_{1}} \right)  \otimes \cdots \otimes \tau_{\U_2}. 
\end{equation*}
The remaining terms are obtained by applying Equation \eqref{eq:boundary_adams_chains} to the first sum in Equation \eqref{eq:differential_sequence_simplices} yeilding:
\begin{equation*}
  \sum_{i=1}^{d}   \tau_{\U_d} \otimes \cdots \otimes  \tau_{\U_i - \{ {u^i}_j\}} \otimes \cdots \otimes \tau_{\U_1}.
\end{equation*}
Taking this sum over all sequences $ \U_1, \ldots, \U_d $ where $d \geq 2$ gives the sum of all terms in $T \left( \partial \sigma_{U} \right)$ whose length is greater than $2$.

To complete the proof, we must show that
\begin{equation*}
\sum_{\U_1, \U_2}   \mu_{2}(\tau_{U_2}, \tau_{U_1}) + \mu_{2}(\tau_{U_1}, \tau_{U_2})   +  \partial \pi_{(\U[1], \U[2])} +  \partial \pi_{(\U[2], \U[1])} 
\end{equation*}
cancels with the component of $T(\partial \sigma_{\U})$ consisting of words of length $1$ in the cyclic bar complex; the first two term above cancel with the first terms in Equation \eqref{eq:contract_adams_chains} applied to $  \partial \pi_{(\U[1], \U[2])}$ and  $ \partial \pi_{(\U[2], \U[1])}  $, the second two terms in  Equation \eqref{eq:contract_adams_chains} are exactly those cancelling with $T( \partial \sigma_{\U})  $, while the last two terms in Equation \eqref{eq:contract_adams_chains} cancel each other after taking the sum over all choices of $\U_1$ and $\U_2$.
\end{proof}

The final result needed for the proof of Lemma \ref{lem:fundamental_class_in_image} is:
\begin{lem}
 The composition
 \begin{equation}
\xymatrix{    C_{-*}(Q)  \ar[r]^(.4){T} &  CC_{*}(C_{-*} \Omega_{q}Q  ) \ar[r]^(.55){\G} & \Lef_{-*}(\sL Q) }
 \end{equation}
is homotopic to the map induced by the inclusion of constant loops in the free loop space.
\end{lem}
\begin{proof}[Sketch of proof]
Recall that the map $\G$ vanishes on all words in the cyclic bar complex of length greater than $2$; in particular, it suffices to consider the components of $T$ consisting of words of length $1$ and $2$.  By construction, the image of a cell $\sigma_{\U}$ under the composition $\G \circ T $ is a sum of cubical chains all of which may be written as the composition of a map to $\Delta^{k}$ with $\sigma_{\U}$.  As this simplex is contractible, we may contract every such chain to a constant loop at its basepoint.  Whenever the basepoint of a loop lies on a boundary facet, we can moreover choose the contraction of a cell to be an extension of one chosen on the boundary.   We conclude that $\G \circ T$ is homotopic to
\begin{equation*}
 \iota_{*} \circ  \ev_{*} \circ \G \circ T
\end{equation*}
where $\ev_{*}$ is the evaluation from the chains of the based loop space to the chains on $Q$, and $\iota_{*}$ is the inclusion of constant loops.  

Since $\G$ agrees on words of length $1$ with the map induced by the inclusion from constant to based loops,  we find that 
\begin{equation*}
  \ev_{*}( \pi_{U_1,U_2})
\end{equation*}
has image a point, and hence vanishes whenever the dimension of $Q$ is greater than $0$ because we are working with normalised chains.  Similarly, as soon as the dimension of $\U_{2}$ is greater than $1$, we have
\begin{equation*}
  \ev_{*}( \tau_{\U_2} \otimes \tau_{\U_1}) = 0
\end{equation*}
because the corresponding cubical chain factors through projection to a cube of dimension $\dim(\U_1)$.  By inspecting Equation \eqref{eq:sum_adams}, we find that the only case where $ \U_{2} $ has dimension $1$ corresponds to 
\begin{equation*}
   \U_{2} = [k,0]  \textrm{ and }   \U_{1} = [0, \ldots, k].
\end{equation*}
The reader can now easily check that the basepoints of $\G(  \tau_{[k,0]} \otimes \tau_{[0, \ldots, k]}  )$ cover the cell $ [0, \ldots, k] $ with multiplicity one.
\end{proof}

\section{A convenient quotient of cubical chains} \label{sec:cuboidal-chains}
Given a topological space $X$, we consider a model for the space of chains which is not standard, and has the convenience of making some of the formulae in the paper relatively simple; in particular, if we were to use the usual cubical chains, the definition of the map $\G$ in Section \ref{sec:an-ad-hoc} would be require non-vanishing terms coming from longer words in the cyclic bar complex.  The purpose of this section is to construct this model and prove that it is chain equivalent to the usual cubical chains.  In the simplicial theory, analogous results are known, though they are usually proved using different techniques (see, for example \cite{barr}).

Recall that the cubical chain complex is the free abelian group generated by maps of cubes to $X$ modulo those which factor through the projection to a factor:
\begin{align*}
 C_{i}(X) & = \frac{\bZ\left[ \Map([0,1]^{i}, X) )   \right] }{\bZ \left[\textrm{degenerate maps}\right]}.
\end{align*}
 The differential is given by the formula 
\begin{equation} \label{eq:differential_cover_generalised}
  \partial \sigma =  \sum_{k=1}^{i} \sum_{\epsilon = 0,1}   \partial_{k,\epsilon} \sigma =  \sum_{k=1}^{i} \sum_{\epsilon = 0,1}   (-1)^{k + \epsilon}  \sigma \circ \delta_{k,\epsilon}
\end{equation}
where $\delta_{k,\epsilon}$ is the inclusion of the face where the $k$\th coordinate is constant and equal to $\epsilon$. 
\begin{defin}
We say that a pair $\sigma_1$ and $\sigma_2$ of generators of the same dimension \emph{fit together} if 
\begin{equation*}
  \sigma_1 \circ \delta_{i,1} =  \sigma_2 \circ \delta_{i,0}.
\end{equation*}
\end{defin}

In this case, given a map $f \co [0,1]^{i-1} \to  [0,1]$ such that 
\begin{align}
  \label{eq:conditions_on_f}
f^{-1}(0) & \subset \{ (t_1, \ldots, t_{i-1}) | \sigma_{1}(t_1, \ldots, t_{i-1}, t_{i})  \textrm{ is independent of } t_{i} \} \\   
f^{-1}(1) & \subset \{ (t_1, \ldots, t_{i-1}) | \sigma_{2}(t_1, \ldots, t_{i-1}, t_{i})  \textrm{ is independent of } t_{i} \}
\end{align}
we define 
\begin{align}
 \label{eq:definition_concatenate}
  \sigma_1 \#_{f}  \sigma_2 \co [0,1]^{i}  & \to X \\
(t_1, \ldots, t_{i-1}, t_i) & \mapsto \begin{cases} \sigma_{1} (t_1, \ldots, t_{i-1},  \frac{t_i}{ f( t_1, \ldots, t_{i-1} ) }) & \textrm{ if } t_{i} \leq f(t_1, \ldots, t_{i-1}) \\
 \sigma_{2} (t_1, \ldots, t_{i-1},  \frac{t_i -  f( t_1, \ldots, t_{i-1} )}{1-  f( t_1, \ldots, t_{i-1} ) } ) & \textrm{ otherwise.}
 \end{cases}
\end{align}
Note that whenever $f(t_1, \ldots, t_{i-1})$ vanishes or is equal to $1$, $   \sigma_1 \#_{f}  \sigma_2$  is still well-defined and continuous at  $  (t_1, \ldots, t_{i-1}, 0 ) $  or at $  (t_1, \ldots, t_{i-1}, 1) $.   

The easiest way to produce cubical chains that fit together is to start  with a cubical chain $\sigma$ and an arbitrary map $ f \co [0,1]^{i-1} \to  [0,1] $.  The graph of $f$ splits the cube $[0,1]^i$ into two halves, which one may think of as 2 families of intervals of varying length parametrised by a cube of dimension $i-1$.    There is a canonical way of mapping $ [0,1]^{i} $  to each of these halves, using the identity in the first $i-1$ factors and ``shrinking'' the last coordinate  to have the appropriate length.  By composing this map with the restriction of $\sigma$ to each half we obtain two chains
\begin{align*}
  \sigma_{1}^{f} (t_1, \ldots, t_{i-1}, t_i)  & = \sigma(t_1, \ldots, t_{i-1}, f(t_1, \ldots, t_{i-1}) t_i)  \\
  \sigma_{2}^{f} (t_1, \ldots, t_{i-1}, t_i)  & = \sigma(t_1, \ldots, t_{i-1}, f(t_1, \ldots, t_{i-1})+  (1- f(t_1, \ldots, t_{i-1}))  t_i)
\end{align*}
which indeed fit together.  
\begin{lem}
The subgroup  \begin{equation} \label{eq:chains_that_fit}
\fit_{*}(X)= \bigoplus  \sigma_1 + \sigma_2 -  \sigma_1 \#_{f}  \sigma_2,
\end{equation}
where the direct sum is taken over all chains which fit together,  is a contractible subcomplex of $ C_{*}(X) $.
\end{lem}
\begin{proof}
An easy computation shows that the map
\begin{equation} \label{eq:map_chains_fit}
  (\sigma, f) \mapsto \sigma -   \sigma_{1}^{f}  -   \sigma_{2}^{f} 
\end{equation}
is a surjection from the chain complex generated by pairs $(\sigma, f)$ with differential
\begin{equation*}
  \partial_{\boxbar} (\sigma,f) = \sum_{k=1}^{i-1} \sum_{\epsilon = 0,1} (-1)^{k + \epsilon}  (\sigma \circ \delta_{k,\epsilon},f \circ  \delta_{k,\epsilon})
\end{equation*}
to $\fit_{*}(X)$.

To prove acyclicity, we consider the map
\begin{align} \label{eq:project_rectangle_interval}
 \fold_{i+1} \co [0,1]^{i+1}  & \to [0,1]^{i}  \\
(t_1,\ldots, t_{i-1}, t_i, t_{i+1}) & \mapsto \begin{cases}(t_1,\ldots, t_{i-1}, t_{i} + t_{i+1}) & \textrm{if } t_i + t_{i+1} \leq 1 \\
 (t_1,\ldots,  t_{i-1}, 1)  & \textrm{otherwise} , \end{cases}
\end{align}
which is the identity on the first $i-1$ factors, and projects the last two factors onto an interval.  Since the terms corresponding to $k=i$ are missing from this differential,  the assignment 
\begin{equation*}
\fold  (\sigma, f) = (-1)^{i}  (\sigma \circ \fold_{i+1}, f \circ \fold_i)
\end{equation*}
defines a null homotopy of the complex generated by pairs $(\sigma,f)$.  To check this, one computes that $\partial_{k,\epsilon}   $ commutes with $\fold$ whenever $k \leq i-1$, and that
\begin{align*}
  \partial_{i,0}  (\sigma \circ \fold_{i+1}, f \circ \fold_i) &  = (-1)^{i} (\sigma, f)  \\
 \partial_{i,1}  (\sigma \circ \fold_{i+1}, f \circ \fold_i) &  = 0.
\end{align*}
 Moreover, an easy computation shows that the kernel of the surjection \eqref{eq:map_chains_fit} is preserved by $\fold$, which implies that $\fit_{*}(X)$ is indeed contractible.
\end{proof}

We also consider the subgroup of $ C_{*}(X)  $
\begin{equation} \label{eq:subcomplex_permute}
\permute_{i}(X) =  \bigoplus_{\substack{\sigma \\ 1 \leq k < i}}  \sigma +  \sigma \circ \phi_{k}
\end{equation}
with $\sigma$ is cubical chain of dimension $i$ and $\phi_{k}$ is the self-homeomorphism of $[0,1]^i$ given by transposing the $k$\th and $k+1$\st factors.  Using the fact that
\begin{equation} \label{eq:cancelling_differential_permute} 
  \partial_{k,\epsilon}   \sigma +   \partial_{k+1,\epsilon}   \sigma  \circ \phi_{k} =  0
\end{equation}
it is easy to check that $ \permute_{*}(X) $ is a subcomplex.
\begin{lem}
 $ \permute_{*}(X) $ is a contractible subcomplex of $ C_{*}(X)  $.
\end{lem}
\begin{proof}
We define a chain complex generated by a cubical cell $\sigma$ and a transposition  $\phi_k$, with differential
\begin{equation} \label{eq:differential_pair_cell_transpo}
  (\sigma, \phi_k) \mapsto  \sum_{\epsilon=0}^{1} \left(  \sum_{j < k}(-1)^{j+\epsilon } ( \sigma \circ \delta_{j,\epsilon}, \phi_{k-1}) +  \sum_{k+1 <i} (-1)^{j+\epsilon } ( \sigma \circ \delta_{j,\epsilon}, \phi_{k}) \right).
\end{equation}

Consider the map
\begin{align*}
  \boxdot \co [0,1]^{2} \times [0,1] & \to [0,1]^2 \\
(t_1,t_2, t_3) & \mapsto \left( \frac{1}{2} + (1-t_3) \left(t_1 - \frac{1}{2}\right) ,  \frac{1}{2} + (1-t_3) \left(t_2 - \frac{1}{2}\right)  \right)
\end{align*}
which gives a homotopy between the identity and the constant map at $\left( \frac{1}{2}, \frac{1}{2}\right) $; given a pair of positive integers $k$ and $i$ such that $1 \leq k \leq i$, we write 
\begin{equation*}
  \boxdot_{k} \co  [0,1]^{i+1} = [0,1]^{i} \times [0,1]  \to [0,1]^{i}
\end{equation*}
for the map which is the identity except on the $k$\th, $k+1$\st, and last factor, where it is given by $\boxdot$.   A short computation shows that the map 
\begin{equation}
  \label{eq:homotopy_permute}
  (\sigma  , \phi_{k} ) \mapsto (-1)^{i+1} (\sigma \circ \boxdot_{k} , \phi_{k} ).
\end{equation}
defines a null homotopy of the differential \eqref{eq:differential_pair_cell_transpo}.  The key point is that
\begin{align*}
  \sigma \circ \boxdot_{k} \circ \delta_{i+1,0}   & =  \sigma \\
   \sigma \circ \boxdot_{k}  \circ \delta_{i+1,1} & = 0
\end{align*}
while  composition with $\boxdot_k$ commutes with the other terms of the differential in Equation \eqref{eq:differential_pair_cell_transpo}.

Moreover, using Equation \eqref{eq:cancelling_differential_permute}, we find that the formula
\begin{equation*}
    (\sigma, \phi_k) \mapsto  \sigma +  \sigma \circ \phi_{k}
\end{equation*}
defines a chain map which surjects to  $ \permute_{*}(X) $, and whose kernel is preserved by the null homotopy.
\end{proof}

\begin{cor} \label{cor:standard_model_chains}
The natural map from the set of cubical chains to the quotient \begin{equation}
  \label{eq:cuboidal_chains_final}
 \Lef_{*}(X) = \frac{ C_{*}(X) }{\fit_{*}(X) +  \permute_{*}(X) }.
\end{equation}
is a quasi-isomorphism. \noproof
\end{cor}

\end{document}